\documentclass[11pt,twosided,reqno]{amsart}

\usepackage{amsmath, amsfonts, amssymb,  mathrsfs,  array, stmaryrd,  indentfirst, amsthm,  hyperref, comment, color, mathtools,bbm}
\usepackage{graphicx, enumitem, tabularx}
\usepackage{verbatim}
\usepackage{tikz}
\usetikzlibrary{calc,decorations.markings}
\usepackage[margin=1.0in]{geometry}
\usepackage{setspace}

\numberwithin{equation}{section}
\theoremstyle{plain}

\newtheorem{thm}{Theorem}[section]
\newtheorem{defn}{Definition}[section]
\newtheorem{prop}{Proposition}[section]
\newtheorem{cor}{Corollary}[section]

\newtheorem{remark}{Remark}[section]

\newcommand{\floor}[1]{\left\lfloor #1 \right\rfloor}

\newcommand{\cC}{\mathcal{C}}

\newcommand{\cJ}{\mathcal{J}}

\newcommand{\cU}{\mathcal{U}}
\newcommand{\cV}{\mathcal{V}}

\newcommand{\pa}{{\partial}}

\newcommand{\ignore}[1]{}

\def \a{\alpha}
\def \b{\beta}
\def \e{\varepsilon}
\def \d{\delta}
\def \1{\mathbf 1}

\def \E{\mathbb{E}}

\def \M{\mathbb{M}}
\def \N{\mathbb{N}}

\def \R{\mathbb{R}}

\newcommand{\comm}[1]{}
\DeclareMathOperator{\arctanh}{arctanh}

\makeatletter
\def\@setcopyright{}
\def\serieslogo@{}
\makeatother

\begin{document}

\title{Finite-Time 4-Expert Prediction Problem}
\author{Erhan Bayraktar} 
\address{Department of Mathematics, University of Michigan}
\email{erhan@umich.edu}
\author{Ibrahim Ekren}
\address{Department of Mathematics, Florida State University}
\email{iekren@fsu.edu}
\author{Xin Zhang} 
\address{Department of Mathematics, University of Michigan}
\email{zxmars@umich.edu}

\date{\today}

\maketitle
\begin{abstract}
We explicitly solve the nonlinear PDE that is the continuous limit of dynamic programming equation \emph{expert prediction problem} in finite horizon setting with $N=4$ experts. The \emph{expert prediction problem} is formulated as a zero sum game between a player and an adversary. By showing that the solution is $\mathcal{C}^2$, we are able to show that the comb strategies, as conjectured in \cite{MR3478415}, form an asymptotic Nash equilibrium. We also prove the ``Finite vs Geometric regret" conjecture proposed in \cite{2014arXiv1409.3040G} for $N=4$, and show that this conjecture in fact follows from the conjecture that the comb strategies are optimal.
\end{abstract}

\begin{singlespace}
 \renewcommand{\contentsname}{\centering Contents}
{\small\tableofcontents}
\end{singlespace}
\onehalfspacing

\section{Introduction}
In this paper, we explicitly solve the degenerate nonlinear PDE with $N=4$
\begin{align}\label{eq:pde}
\pa_t u^T(t,x)+\frac{1}{2}\sup_{J \in P(N)}e_{J}^\top \pa^2_{xx} u^T(t,x)e_J=0, \notag\\
u^T(T,x)=\Phi(x):=\max_i{x_i},
\end{align}
where $P(N)$ is the power set of $\{1, \dotso, N\}$ and $e_J:=\sum_{j\in J}e_j$ with $\{e_j\}_{j \in \{1, \cdots, N\}}$ representing the standard basis of $\mathbb{R}^N$. 
Kohn and Drenska \cite{MR3768426,Drenska2019} showed that this equation has a unique viscosity solution, which is the continuous limit of dynamic programming equation of the Expert Prediction Problem with finite stopping. The Expert Prediction Problem is a zero sum game between a player and an adversary (see e.g. \cite{MR3478415}). Here we construct this unique viscosity solution explicitly 
\begin{align}
u^T(t,x)=&\frac{-1}{16\sqrt{2}}\int_{-\infty}^{\infty} \frac{e^{-(T-t)r^2}}{r^2}\left(\psi\left({{r}\theta\cdot x^o+\frac{\pi}{2}}\right)\sum_{k=1}^4\cos\left({r}{\a_k \cdot x^o}\right)-4
\right. \notag \\
& \left. - \psi\left({{r}\theta\cdot x^o}\right)\sum_{k=1}^4\sin\left({r}{\a_k \cdot x^o}\right)\right) dr +\frac{1}{4}\sum_{i=1}^4x_i +\frac{1}{2}\sqrt{\frac{(T-t)\pi}{2}}, \label{eq:solution}
\end{align}
where $\psi$ is the $2\pi$ periodic square wave function, $x^o$ is obtained from rearranging the coordinates of $x$ in the increasing order, and $\a_k,\theta\in \R^4$ are defined by $\a_{k,j}=\frac{3}{\sqrt{2}}\1_{\{k= j\}}-\frac{1}{\sqrt{2}}\1_{\{k\neq j\}}$, $\theta=\frac{1}{\sqrt{2}}(1,1,-1,-1)$.
We show that $u^T \in \mathcal{C}^2$, and due to this regularity, we are able to show that the balanced comb strategy and the probability matching algorithm proposed in \cite{MR3478415} are the asymptotic saddle points for the game. As noted in \cite{MR3768426}, in particular for $x=0, t=0, T=1$, the value $u^1(0,0)$ provides the expansion of the best regret as 
$$V^M(0,0)=u^1(0,0)\sqrt{M}+o(\sqrt{M})\mbox{ as }M\to \infty,$$
where $V^M$ is the value function of expert prediction problem with time maturity $M$. According to our solution \eqref{eq:solution}, we obtain the explicit value of the first order coefficient $u^1(0,0)=\frac{1}{2} \sqrt{\frac{\pi}{2} }$, which resolves the open problem in \cite{2014arXiv1409.3040G} for $N=4$; see also \cite{NIPS2017_6896}.

Prediction problem with expert advice is classical and fundamental in the field of \emph{machine learning}, and has been studied for decades. We refer the reader to  \cite{MR2409394} for a nice survey. It is a dynamic zero-sum game between a player and an adversary. At each of the $M$ rounds, based on all the prior information,  the player chooses one of the $N$ experts to follow, and simultaneously the adversary chooses a set of winning experts. The increment of the gain for each expert is either $0$ or $1$ depending on whether the expert is chosen by the adversary, and the increment of the gain of the player is that of the expert the player follows. Given a fixed maturity $M$, the objective of the player is to minimize the regret $\max\limits_{i} G_M^i - G_M$, while the adversary wants to maximize the regret, where $G_M^i$ and $G_M$ are the gain of the expert $i$ and the player, respectively.

For the case of $2$ experts,  Cover  \cite{MR0217944} showed that the asymptotically optimal strategy for the adversary is the one that chooses an expert uniformly at random. For the case of $3$ experts with geometric stopping, Gravin, Peres and Sivan \cite{MR3478415} showed that the comb strategy, which chooses the experts with the highest gain and the one with the lowest gain with probability $\frac{1}{2}$, and chooses the second leading expert with probability $\frac{1}{2}$ is asymptotically optimal for the adversary.  They also showed that  the probability matching algorithm, which consists of  following an expert with the probability that under the comb strategy that that expert will be the leading one at the end of game, is the player's asymptotically optimal response. For the case of $N=3$ experts with finite stopping, it has been shown in \cite{NIPS2017_6896} that the comb strategy is asymptotically  optimal. While both \cite{NIPS2017_6896,MR3478415} use the theory of random walk, \cite{Drenska2019} exploits the power of the PDE method. By considering a scaled game, they have shown that the value function of discrete games converges to the viscosity solution of a PDE. Following this setting, for the case of $N=4$ experts in the geometric horizon setting, Bayraktar, Ekren and Zhang \cite{2019arXiv190202368B} showed that the comb strategy is asymptotically optimal by explicitly solving the corresponding nonlinear PDE. And very recently in \cite{2019arXiv191101641K}, Kobzar, Kohn and Wang found lower and upper bounds for the optimal regret for finite stopping problem by constructing certain sub- and supersolutions of \eqref{eq:pde} following the method of \cite{DBLP:journals/corr/Rokhlin17}. Their results are only tight for $N=3$ and improved those of \cite{NIPS2017_6896}. Let us also mention the Multiplicative Weights Algorithm, which is asymptotically optimal as both $N, M \to \infty$ (see \cite{MR1470152}).

In this paper we construct an explicit solution to \eqref{eq:pde} for $N=4$ with finite stopping. We build our candidate solution based on the conjecture of \cite{MR3478415}, which states that the comb strategy is asymptotically optimal for any number of experts in both finite and geometric horizon problem. Note that if the comb strategy is asymptotically optimal, the solution to \eqref{eq:pde} should also satisfy a linear PDE with comb strategy based coefficients (see \eqref{eq:pdegeolinear}), which is shown to be true in the geometric horizon setting in \cite{2019arXiv190202368B}. The key observation is that the PDE of the finite horizon case can, at least heuristically, be obtained by applying the inverse Laplace transform to the solution of \cite{2019arXiv190202368B} extended to the complex plane. This is at a heuristic level because these linear PDEs, unlike \eqref{eq:pde}, may not have unique solution  and the analytic extension of our function to the complex plane is not well-behaved. In Appendix \ref{appendix}, we  perform this formal inverse Laplace transform and obtain the explicit expression in \eqref{eq:solution}. We show in Theorem \ref{thm:propertiesu} that \eqref{eq:solution} is the classical solution of \eqref{eq:pdegeolinear}. In Theorem \ref{thm:optimality}, we show that it also satisfies \eqref{eq:pde} by verifying that the comb strategy is optimal for the limiting problem. In Theorem \ref{thm:asyopt}, we show that the probability matching strategy for the player and the comb strategy for the adversary form an asymptotic saddle point, resolving the conjecture of \cite{MR3478415} for four experts. As a corollary, we resolve the Finite versus Geometric regret conjecture in \cite{2014arXiv1409.3040G} (see also \cite{NIPS2017_6896}); see Corollary \ref{cor:cor}. Our work reveals that the ratio of the value of two problems (which was conjectured to be $\frac{2}{\sqrt{\pi}}$) actually comes from the inverse Laplace transform; see \eqref{eq:gamma}. We also apply our method to obtain an explicit expression for $u^T$ in the 3 experts case, which was not known.

We now detail some of the difficulties in our proofs. The first main difficulty is showing that the boundary condition $u^{T}(T,x)=\Phi(x)$ is satisfied. We first write the function $u^T$ in terms of sine and cosine integral functions (see \cite{MR1225604}) and perform some intricate and long arguments from complex analysis relying on the properties of these functions. Second main difficulty is showing that the function $u^T$ actually solves the nonlinear PDE. We perform this analysis through a verification type of argument, in which we show that certain inequalities are satisfied for all $(t,x)$ and hence ruling out all the other alternative strategies for the adversary. This analysis is the most demanding part of the paper in which we rely on the properties of the Jacobi-theta function (see \cite{MR2723248}) and other properties of Fourier series. The third main difficulty is showing that the probability matching algorithm for the player and the comb strategy for the adversary form an asymptotic saddle point.  Relying on some delicate estimates, we show that the value function of discrete game converges to $u^T$ if either the player adopts the probability matching algorithm, or the adversary adopts the comb strategy.

The rest of the paper is organized as follows. In Section \ref{s:Prelim}, we introduce the problem and provide some of lemmata. In Section \ref{s:mainresults}, we state the three main results of our paper, namely Theorem \ref{thm:propertiesu}, \ref{thm:optimality} and \ref{thm:asyopt}. Here we also state the Corollary \ref{cor:cor} which resolves the ``geometric versus finite horizon conjecture" for 4 experts. In Section \ref{s.proof}, we provide all the proofs, and in Appendix \ref{appendix}, we provide a heuristic derivation of the value functions for $N=3,4$ via inverse Laplace transform. 

In the rest of this section, we will provide some frequently used notation.

\noindent {\bf{Notation.}} Denote the left hand side and the right hand side derivatives by $\pa^-, \pa^+$ respectively. Denote the number of experts by $N$, the time horizon of the discrete game by $M$, and the time horizon of the continuous time control problem by $T$ (so $M$ in our paper represents the $T$ in \cite{NIPS2017_6896,MR3478415}).  Denote by $U$ the set of probability measures on $\{1,\ldots, N\}$ and by $V$ the set of probability measures on $P(N)$, the power set of $\{1,\ldots N\}$. We denote by $\{e_i\}_{i=\{1,\ldots, N\}}$ the canonical basis of $\R^N$, and for $J\in P(N)$, $e_J$ is defined as
$e_J:=\sum_{j\in J}e_j$. For all $x\in \R^N$, we denote by $x_i$ the $i$-th coordinate of $x$, by $\{x^{(i)}\}_{i=1,\ldots, N}$ the ranked coordinates of $x$ with 
$x^{(1)}\leq x^{(2)}\leq \ldots\leq x^{(N)}$, by $\{i_1, \dotso, i_N\}$ the reordering of $\{1, \dotso, N\}$ such that $x_{i_1} \leq x_{i_2} \leq \dotso \leq x_{i_N}$ with the convention that if two components $x_i$ and $x_j$ are equal and $i<j$ then the ordering is defined to be $x_i \leq x_j$. We define $x^o:=\left(x^{(1)}, \dotso, x^{(N)} \right)$. 

\section{Preliminaries}\label{s:Prelim}

We assume that a player and an adversary are playing a zero-sum game, and they interact through the evolution of the gains of $N$ experts.  At step $m \in \N$, by $\{G_k^i\}_{k=1, \dotso, m-1}$, we denote the history of the gains of each expert $i=1, \dotso N$, and by $\{G_k\}_{k=1, \dotso, m-1}$, the history of the gains of the player. After observing all the prior history $\mathcal{G}_{m-1}:=\{(G_k^i,G_k): 1\leq i \leq N, 1 \leq k \leq m-1\}$,  simultaneously, the adversary chooses some experts $J_m \in P(N)$, and the player chooses the expert $I_m \in \{1, \dotso, N\}$ to follow. For each $i=1, \dotso, N$, the gain of expert $i$ increases by $1$ if he is chosen by the adversary, otherwise remains the same. The increment of the player's gain follows that of the expert $I_m$ he chooses. Therefore we have 
\begin{align*}
G^i_m&=G^i_{m-1}+\mathbbm{1}_{\{i \in J_m \}}, \quad \quad  i=1, \dotso, N; \\
G_m&=G_{m-1}+\mathbbm{1}_{\{I_m \in J_m \}}.
\end{align*}

In order to have a value for the game, we allow both the adversary and the player to adopt randomized strategies. At step $m \in \N$, the adversary decides on the distribution $\beta_m \in V$ to draw $J_m$ from, and independently the player decides on the distribution $\a_m \in U$ of $I_m$. Then the dynamic of $\{(G^i_m,G_m: 1 \leq i \leq N\}$ is given by 
\begin{align*}
\E^{\a_m, \b_m}[G^i_m| \mathcal{G}_{m-1}]&=G^i_{m-1}+\sum\limits_{J \in P(N)} \beta_m(J) \mathbbm{1}_{ \{i \in J\}},    \quad \quad  i=1, \dotso, N; \\
\E^{\a_m, \b_m}[G_m| \mathcal{G}_{m-1}]&=G_{m-1}+\sum\limits_{i=1}^N \sum\limits_{J \in P(N)} \a_m(i) \beta_m(J) \mathbbm{1}_{ \{i \in J\}}.
\end{align*}

Denote by $\cU$ the collection of sequences $\{\a_m\}_{ m \in \N}$ such that $\a_m$ is a function of $\mathcal{G}_{m-1}$, by $\cV$ the collection of such sequences $\{\b_m\}_{m \in \M}$. We take
\begin{align}\label{eq:X}
X_m:=(X^1_m,\ldots, X_m^N):=(G^1_m-G_m,\ldots, G^N_m-G_m),
\end{align}
the difference between the gain of the player and the experts. Define the function 
\begin{align*}
\Phi: x \mapsto \max\limits_{ 1 \leq i \leq N}x_i=x^{(N)},
\end{align*}
and the regret of the player at step $m \in \N$, 
\begin{align*}
\Phi(X_m)=\max_{i=1,\ldots,N}G^i_m-G_m.
\end{align*}

The objective of the player is to minimize his expected regret at maturity $M$ while the objective of the adversary is to maximize the regret of the player. By the Minimax theorem, the game has a solution (see \cite{MR3768426,MR3478415}), i.e., 
\begin{align}\label{eq:defVd}
\sup_{\b\in\cV}\inf_{\a\in \cU} \E^{\a,\b}\left[\Phi(X_M)|X_0=x\right]=\inf_{\a\in \cU}\sup_{\b\in\cV}  \E^{\a,\b}\left[\Phi(X_M)|X_0=x\right],
\end{align}
where $\E^{\a,\b}$ is the probability distribution under which we evaluate the regret given the controls $\a=\{\a_m\}$ and $\b=\{\b_m\}$. Therefore we can define the value function 
\begin{align*}
V^M(m,x):=\sup_{\b\in\cV}\inf_{\a\in \cU} \E^{\a,\b}\left[\Phi(X_M)|X_m=x\right]=\inf_{\a\in \cU}\sup_{\b\in\cV}  \E^{\a,\b}\left[\Phi(X_M)|X_m=x\right],
\end{align*}
which satisfies the following dynamical programming principle
\begin{align*}
& V^M(m,x)=\inf_{\a\in U} \sup_{\beta\in V}\sum_{J} \beta_J \left( V^{M}\left(m+1,x+e_J\right)-\a(J)\right).
\end{align*}

 Additionally, it was shown in \cite{MR3768426} that for any sequence $m_M\in \N$ and $x_{m_M} \in \R^4$  such that $\frac{m_M T}{M}\to t$ and $\frac{x_{m_M}\sqrt{T}}{\sqrt{M}} \to x$ as $M \to\infty$, we have that 
$$\lim\limits_{M \to \infty} \frac{{V^{M}(m_M,x_{m_M})}\sqrt{T}}{\sqrt{M}}\to u^T(t,x),$$
where $u^T(t,x)$ is the unique viscosity solution to \eqref{eq:pde}. Also, we have the Feynmann Kac representation of $u^T(t,x)$
\begin{align}\label{eq:fk}
u^T(t,x)=\sup_{\sigma} \E\left[\Phi(X^{\sigma}_T)|X_t=x\right],
\end{align}
where $X^{\sigma}$ is defined by $X_u=X_t+ \int_t^u \sigma_s dW_s$ with $W$ a $1$-dimensional Brownian motion and the progressively measurable process $(\sigma_s)$ satisfying for all $s \in [t,u]$, $\sigma_s \in \{e_J: J \in P(N) \}$.

\section{Main Results}\label{s:mainresults}
\subsection{Solution to PDE \eqref{eq:pde} with $N=4$}
Define $\a_k,\theta\in \R^4$ by $\a_{k,j}=\frac{3}{\sqrt{2}}\1_{\{k= j\}}-\frac{1}{\sqrt{2}}\1_{\{k\neq j\}}$ and 
$\theta=\frac{1}{\sqrt{2}}(1,1,-1,-1)$. 
Denote the $2\pi$ periodic square wave function by 
\begin{align*}
\psi(r):=sign\left(\tan\left(\frac{{r}}{2}\right)\right)=sign\left(\sin\left({{r}}\right)\right).
\end{align*}
Define the auxiliary function 
\begin{align*}
\Lambda(r,x):= &\left(\psi\left({{r}\theta\cdot x+\frac{\pi}{2}}\right)\sum_{k=1}^4\cos\left({r}{\a_k \cdot x}\right)-4
- \psi\left({{r}\theta\cdot x}\right)\sum_{k=1}^4\sin\left({r}{\a_k \cdot x}\right)\right),
\end{align*}
and our conjectured solution to \eqref{eq:pde}
\begin{align}\label{eq:defu}
 u^T(t,x):=\frac{-1}{16\sqrt{2}}\int_{-\infty}^{\infty} \frac{e^{-(T-t)r^2}}{r^2}\Lambda(r,x^o) dr +\frac{1}{4}\sum_{i=1}^4x_i +\frac{1}{2}\sqrt{\frac{(T-t)\pi}{2}}.
\end{align}

\begin{remark}
Due to the presence of $r^{-2}$, there is a possible integrability issue of $$\int_{-\infty}^{\infty} \frac{e^{-(T-t)r^2}}{r^2}\Lambda(r,x^o) dr.$$ However
as a result of  the fact that $\sum_{k=1}^4\a_k \cdot x^o=0$,  we have the Taylor expansion around $0$
$$\Lambda(r,x^o)=\sum_{k=1}^4 |r|\a_k \cdot x^o-\sum_{k=1}^4\frac{(r\a_k \cdot x^o)^2}{2}+o(r^2)=O(r^2).$$
Thus, $u^T(t,x)$ is well-defined. 
\end{remark}
\begin{remark}
Since the function  $\Lambda(r,x)$ is even with respect to $r$,  we sometimes use the expression
\begin{align*}
u^T(t,x)=\frac{-1}{8\sqrt{2}}\int_{0}^{\infty} \frac{e^{-(T-t)r^2}}{r^2}\Lambda(r,x^o) dr +\frac{1}{4}\sum_{i=1}^4x_i +\frac{1}{2}\sqrt{\frac{(T-t)\pi}{2}}.
\end{align*}
\end{remark}

\begin{defn}\label{def:comb}
For all $x\in \R^4$ with $x_{i_1}\leq x_{i_2}\leq x_{i_3}\leq x_{i_4}$, we denote by $\cJ_\cC(x)\in P(4)$ the comb strategy which chooses the experts $i_4$ and $i_2$. Denote $\sigma_{\cC}(X_s):=e_{\cJ_\cC(X_s)}$ to be the corresponding control of problem \eqref{eq:fk}. We take the convention that if two components $x_i$ and $x_j$ of the points are equal for $i<j$ then the ordering of the point is taken with $x_i\leq x_j$.
\end{defn}

The following theorem assembles properties of $u^T$, and its proof is provided in Section \ref{ss.proof1}.
\begin{thm}\label{thm:propertiesu}
The function $u^T$ is symmetric in $x$, satisfies $u^T\in \mathcal{C}([0,T]\times \R^4)\cap \mathcal{C}^2([0,T)\times \R^4)$ and
\begin{align}\label{eq:pdegeolinear}
\pa_t u^T(t,x)+\frac{1}{2}e_{\cJ_\cC(x)}^\top \pa^2_{xx} u^T(t,x)e_{\cJ_\cC(x)}=0&, \notag \\
u(T,x)=\max_{i=1,\dots, 4}x_i& .
\end{align}
The first derivative of $u^T$ on $\theta \cdot x^{o}<0$ is 
\begin{align}\label{1st}
\pa_{x_i}u^T(t,x)=&\frac{1}{16\sqrt{2} }\int_{-\infty}^{\infty} \frac{e^{-(T-t)r^2}}{r}\sum_{k=1}^4\a_{k,i}\left(\psi\left({{r}\theta\cdot x^{o}+\frac{\pi}{2}}\right)\sin\left({r}{\a_k \cdot x}\right) \right. \notag \\
 &\left. + \psi\left({{r}\theta\cdot x^{o}}\right)\cos\left({r}{\a_k \cdot x}\right)\right)dr+\frac{1}{4},
\end{align}
and if $\theta \cdot x^{o}=0$, it is
\begin{align}\label{1zero}
\pa_{x_i}u^T(t,x)=\frac{1}{4}.
\end{align}
If $\theta \cdot x^{o}<0$, and $x^{(2)} < x^{(3)}$, we have 
\begin{align}\label{2nd}
\pa^2_{x_ix_j}u^T(t,x)=& \frac{1}{16\sqrt{2} }\int_{-\infty}^{\infty} {e^{-(T-t)r^2}}\sum_{k=1}^4\a_{k,i}\a_{k,j}\left(\psi\left({{r}\theta\cdot x^{o}+\frac{\pi}{2}}\right)\cos\left({r}{\a_k \cdot x}\right) \right. \notag \\
& \left. - \psi\left({{r}\theta\cdot x^{o}}\right)\sin\left({r}{\a_k \cdot x}\right)\right)dr\notag\\
&+\frac{\pa_{x_j} (\theta \cdot x^{o})}{16 \sqrt{2}} \sum_{l\in \mathbb{Z}}\frac{(-1)^le^{- \frac{(T-t)(\pi (l+1/2))^2}{(\theta\cdot x^{o})^2}}}{\theta\cdot x^{o}}\sum_{k=1}^4 2 \a_{k,i}\sin \left(\frac{\a_k \cdot x \pi(l+1/2)}{\theta\cdot x^{o}}\right)\notag\\
&-\frac{\pa_{x_j} (\theta \cdot x^{o})}{16\sqrt{2}} \sum_{l\in \mathbb{Z}}\frac{(-1)^le^{- \frac{(T-t)(\pi l)^2}{(\theta\cdot x^{o})^2}}}{\theta \cdot x^{o}}\sum_{k=1}^4 2\a_{k,i}\cos \left(\frac{\a_k \cdot x \pi l}{\theta\cdot x^{o}}\right),
\end{align}
if $\theta \cdot x^{o}<0$ and $x^{(2)} = x^{(3)}$,
\begin{align}\label{3rd}
\pa^2_{x_ix_j}u^T(t,x)=& \frac{1}{16\sqrt{2} }\int_{-\infty}^{\infty} {e^{-(T-t)r^2}}\sum_{k=1}^4\a_{k,i}\a_{k,j}\left(\psi\left({{r}\theta\cdot x^{o}+\frac{\pi}{2}}\right)\cos\left({r}{\a_k \cdot x}\right) \right. \notag \\
& \left. - \psi\left({{r}\theta\cdot x^{o}}\right)\sin\left({r}{\a_k \cdot x}\right)\right)dr,  
\end{align}
and if $\theta \cdot x^{o}=0$, 
\begin{align}\label{2zero}
\pa^2_{x_ix_j}u^T(t,x)&=\frac{1}{16\sqrt{2} }\int_{-\infty}^{\infty} {e^{-(T-t)r^2}}\sum_{k=1}^4\a_{k,i}\a_{k,j}dr.
\end{align}
\end{thm}

The proof of the following theorem is in Section \ref{ss:proof2}.
\begin{thm}\label{thm:optimality}
The function $u^T$ defined in \eqref{eq:defu} is also a solution to \eqref{eq:pde} and the comb strategy $e_{\cJ_\cC}$ is optimal for the problem \eqref{eq:pde}.
\end{thm}

\subsection{An asymptotical Nash equilibrium for the game \eqref{eq:defVd} with $N=4$}

Given the value of $u^T$, we now describe a family of asymptotically optimal strategies for both the player and the adversary. 
Inspired by \cite{MR3478415} we give the following definition. 
\begin{defn}\label{def:controls}
(i)For $M \in \N$, we denote  by $\cJ^b_\cC(M)\in \cV$, the balanced comb strategy, which at state ${x}\in \R^4$ and round $m \in \N$, chooses experts $\cJ_\cC\left( x \right)\in P(4)$ with probability $\frac{1}{2}$ and $\cJ^c_\cC\left(x \right)\in P(4)$ with probability $\frac{1}{2}$.
 
 (ii)For $M \in \N$, we denote by $\a^*(M)\in \cU$, the strategy that, at state ${x}\in \R^4$ and round $m \in \N$, chooses the expert $i$ with probability $\pa_{x_i}u^T\left(\frac{mT}{M},\frac{x\sqrt{T}}{\sqrt{M}}\right)$ for all $i=1,\dots, 4$. \end{defn}

\begin{remark}
Note that Definition \ref{def:comb} defines a control for the problem \eqref{eq:fk} while Definition \ref{def:controls} defines controls for the game \eqref{eq:defVd}. Hence the latter depends on $M,T$ and $x$,  and the control $\a^*(M)$ actually reflects the scaling between the two problems (see \cite{Drenska2019} for details).
\end{remark}

\begin{remark}
According to the Feynmann Kac representation \eqref{eq:fk} and Theorem \ref{thm:propertiesu}, we have $$u^T(t,x)=\E\left[\Phi\left(x + \int_t^T \sigma_{\cC}(X_s) dW_s \right) \right].$$
Then heuristically 
\begin{align*}
\pa_{x_i} u^T(t,x)=\E\left[\pa_{x_i}\Phi\left(x + \int_t^T\sigma_{\cC}(X_s) dW_s \right) \right]=\mathbb{P}\left[(X_T^{\sigma_{\cC}})_i=(X_T^{\sigma_{\cC}})^{(4)} | X_t=x\right],
\end{align*}
which is just the probability matching algorithm proposed in \cite{MR3478415}.
\end{remark}

\begin{defn}
Define the following two value functions
\begin{align*}
\underline V^M(m,.)&: x \mapsto\inf_{\a\in \cU}\E^{\a,\cJ^b_\cC(M)}\left[{\Phi(X_M)}|X_m=x\right],\\
\overline V^M(m,.)&:x \mapsto \sup_{\b\in \cV}\E^{\a^*(M),\b}\left[{\Phi(X_M)}|X_m=x\right],
\end{align*}
and their limits 
\begin{align*}
\underline u^T(t,x):=\liminf\limits_{(M,\frac{m_M T}{N},\frac{x_{m_M} \sqrt{T}}{\sqrt{M}} )\to (\infty,t,x)}  \frac{\underline V^M\left(m_M,x_{m_M} \right) \sqrt{T}}{\sqrt{M}}, \\
\overline u^T(t,x):=\limsup\limits_{(M,\frac{m_M T}{M},\frac{x_{m_M} \sqrt{T}}{\sqrt{M}}) \to (\infty,t,x)}   \frac{ \overline V^M(m_M,x_{m_M}) \sqrt{T}}{\sqrt{M}}.
\end{align*}
\end{defn}

The proof the following theorem can be found in Section \ref{ss:asymptotic}.
\begin{thm}\label{thm:asyopt}
The family of strategies $(\a^*(M))_{M\in \N}\in \cU^\N$ and $(\cJ^b_\cC(M))_{M\in \N}\in \cV^\N$ are asymptotic saddle points for the player and the adversary,  in the sense that for all $(t,x)\in[0,T]\times \R^4$
\begin{align*}
\underline u^T(t,x)= \overline u^T(t,x)=u^T(t,x).
\end{align*}
\end{thm}

It can be easily seen that $\underline u^T(t,x) \leq u^T(t,x) \leq \overline u^T(t,x)$, and our main result states that they are actually equal, which implies that at the leading order it is optimal for both the player and the adversary to choose respectively the controls $\a^*(M)$ and $\cJ_\cC^b(M)$, i.e., for any 
$\a_M\in \cU$, $\beta_M \in \cV$ and $T>0$, we have that 
$$\liminf_{M \to \infty}\sqrt{\frac{T}{M} }\left(\E^{\a_M,\cJ^b_\cC(M)}\left[{\Phi(X_M)} \left|X_0=\frac{\sqrt{M}x}{\sqrt{T}}\right.\right]- \E^{\a^*(M),\cJ^b_\cC(M)}\left[{\Phi(X_M)} \left|X_0=\frac{\sqrt{M}x}{\sqrt{T}}\right.\right]\right)\geq 0,$$
\begin{align*}
\limsup_{M \to \infty}\sqrt{\frac{T}{M} }\left( \E^{\a^*(M),\b_M}\left[{\Phi(X_M)} \left|X_0=\frac{\sqrt{M}x}{\sqrt{T}}\right.\right]- \E^{\a^*(M),\cJ^b_\cC(M)}\left[{\Phi(X_M)} \left|X_0=\frac{\sqrt{M}x}{\sqrt{T}}\right.\right]\right)\leq 0.
\end{align*}

\subsection{Relation between the finite and geometric stopping}\label{ss:FvsG}
We recall the following results from \cite{2019arXiv190202368B} and \cite{Drenska2019}.
Let $T^\d$ be a geometric random variable with parameter $\d>0$. Define 
\begin{align*}
V^\d(X_0):=\sup_{\b\in\cV}\inf_{\a\in \cU} \E^{\a,\b}\left[\Phi(X_{T^\d})\right]=\inf_{\a\in \cU}\sup_{\b\in\cV}  \E^{\a,\b}\left[\Phi(X_{T^\d})\right],
\end{align*}
and
\begin{align*}
u^\d:x\in \R^N\mapsto V^\d\left(\frac{x}{\sqrt{\d}}\right)\sqrt{\d}.
\end{align*}
so that as $\d\downarrow 0$, the function $u^\d$ converges locally uniformly to $u:\R^N\mapsto \R$ which is the unique viscosity solution of the equation
\begin{align}\label{eq:pdegeo2}
u(x)-\frac{1}{2}\sup_{J\in P(N)} e_J^\top \pa^2 u(x)e_J= \Phi(x).
\end{align}

The main conjecture in \cite{MR3478415} regarding the relation between the finite and geometric horizon control problems is that 
\begin{align*}V^{M}(0,0)\underset{M \to+\infty}{\sim}\frac{2}{\sqrt{\pi}}V^\frac{1}{M}(0).
\end{align*}
The corollary below shows that this statement is true for $N=3,4$.

\begin{cor}\label{cor:cor}
For $N=3,4$, we have the limit $$\lim\limits_{M \to \infty} \frac{V^M(0,0)}{V^{\frac{1}{M}}(0)}=\frac{2}{\sqrt{\pi}}.$$
\end{cor}
\begin{proof}
According to Theorem \ref{thm:optimality} and Proposition \ref{prop:solution}, \eqref{eq:defu} and \eqref{eq:parobalic3} are solutions to \eqref{eq:pde} with $N=4$ and $N=3$, respectively. 
As a result of \cite[Proposition 6.1]{2019arXiv190202368B} and \cite[Theorem 8]{Drenska2019}, \eqref{eq:u} and \eqref{eq:elliptic3} are the solutions to \eqref{eq:pdegeo2} with $N=4$ and $N=3$, respectively. Plugging in $T=1,t=0,x=0$ into these equations, we obtain that for $N=4$, $u^1(0,0)=\frac{1}{2}\sqrt{\frac{\pi}{2} }, \ u(0)= \frac{\pi }{ 4\sqrt{2} }$, and for $N=3$,
$u^1(0,0)=\frac{4}{3\sqrt{2\pi}}, \ u(0)= \frac{4 }{ 6\sqrt{2} }. $
Due to the equalities 
\begin{align*}
\lim\limits_{M \to \infty} \frac{1}{\sqrt{M}}V^M(0,0)=u^1(0,0), \quad\lim\limits_{M \to \infty} \frac{1}{\sqrt{M}}V^{\frac{1}{M}}(0)=u(0),
\end{align*}
we conclude that for both $N=3$ and $N=4$, 
\begin{align*}
\lim\limits_{M \to \infty} \frac{V^M(0,0)}{V^{\frac{1}{M}}(0)}=\frac{u^1(0,0) }{u(0)}=\frac{2}{\sqrt{\pi}}.
\end{align*}
\end{proof}
\subsubsection{From ``optimality of the comb strategy conjecture'' to ``Finite vs Geometric regret conjecture''}

Let us start by recalling that for any $T>0$ and $(t,x)\in [0,T]\times \R^N$ the equation 
 \begin{align}\label{eq:sdecomb}
 X^{t,x}_u=x+\int_t^u \sigma_{{\cC}}(X^{t,x}_s)dW_s, \quad \mbox{ for }u\in [t,T],
 \end{align}
 has a unique weak solution (see \cite[Theorem 2.1]{bass1987uniqueness}).
\begin{prop}\label{conject}
Let $N\geq 2$ and assume that the comb strategies are optimal in the sense that the weak solution of \eqref{eq:sdecomb} is an optimizer of 
\eqref{eq:fk} and $u^T$ is $C^{0}([0,T]\times \R^N)\cap C^{1,2}([0,T)\times \R^N)$ and satisfies for some $\e>0$ and for all $x\in \R^N$
\begin{align}\label{eq:checkconject}
\int_0^\infty e^{-T}\sup_{|x-y|\leq \e}|\pa^2_{xx}u^T(0,y)| dT<\infty.
\end{align}
Then, the comb strategy is optimal for the problem \eqref{eq:pdegeo2} and the function $u$ defined at \eqref{eq:pdegeo2} satisfies
\begin{align}\label{eq:laplace}
u(x)=\E\left[\int_0^\infty e^{- T} \Phi(X^{0,x}_T)dT\right]=\int_0^\infty e^{- T}u^T(0,x)dT. 
\end{align}
\end{prop}
\begin{remark}
Given the results in Proposition \ref{conject}, a simple change of variable formula allows us to claim that the function 
 $$u^\lambda (x)=\lambda^{-3/2}u(\sqrt{\lambda }x)$$
solves the equation
$$\lambda u^\lambda(x)-\frac{1}{2}\sup_{J\in P(N)} e_J^\top \pa^2 u^\lambda (x) e_J=\Phi(x)$$
and satisfies
$$u^\lambda(x)=\E\left[\int_0^\infty e^{- \lambda T} \Phi(X^{0,x}_T)dT\right]=\int_0^\infty e^{- \lambda T}u^T(0,x)dT$$
Therefore, a corollary of \eqref{eq:laplace} is the following relationship due to the Inverse Laplace transform from $u^\lambda(x)$ to $u^T(0,x)$,
\begin{align}\label{eq:gamma}
u^1(0,0)=\frac{u(0)}{2\pi i}\int_{1-i\infty}^{1+i\infty} {e^\lambda}{\lambda^{-3/2}}d\lambda =-\frac{\Gamma\left(-\frac{1}{2}\right)}{\pi}u(0)=\frac{2}{\sqrt{\pi}}u(0),
\end{align}
where $\Gamma$ is the gamma function. Thus, under the assumption of the optimality of the comb strategies for the finite time problem and some technical assumption the Proposition \ref{conject} yields the constant in the ``Finite versus Geometric" conjecture of \cite{2014arXiv1409.3040G} for all $N$; see also \cite{NIPS2017_6896}.
According to \eqref{2nd}, we have $$|\pa^2_{x_ix_j} u^T(0,x)| \leq C \left(\int_{-\infty}^{\infty} e^{-Tr^2} dr + \sum\limits_{l  \in \mathbb{Z}} \frac{e^{\frac{-T (\pi l)^2}{4(\theta \cdot x^o)^2}} }{\theta \cdot x^o}\right),$$ 
where $C$ is a positive constant. Multiplying both sides by $e^{-T}$ and integrating from $0$ to $\infty$, we can easily check \eqref{eq:checkconject} for our expression \eqref{eq:defu}. As a result, Proposition \ref{conject} in fact implies Theorems 3.1 and 3.2 of \cite{2019arXiv190202368B}.

\end{remark}

\section{Proofs}\label{s.proof}

\subsection{Proof of Theorem \ref{thm:propertiesu}}\label{ss.proof1}

\subsubsection{Continuity of $x \mapsto u^T(t,x)$}
\begin{proof}
Using \eqref{eq:defu} and the continuity of $x \mapsto x^{o}$, it suffices to show that  
\begin{align*}
\Lambda(r,x)= \left(\psi\left({{r}\theta\cdot x+\frac{\pi}{2}}\right)\sum_{k=1}^4\cos\left({r}{\a_k \cdot x}\right)-4
- \psi\left({{r}\theta\cdot x}\right)\sum_{k=1}^4\sin\left({r}{\a_k \cdot x}\right)\right).
\end{align*} 
is continuous with respect to $x$. Due to the formula  $\sin(x)+\sin(y)=2\sin(\frac{x+y}{2})\cos(\frac{x-y}{2})$, and the fact $\sum\limits_{k=1}^4 \alpha_k \cdot x=0$, we obtain 
\begin{equation*}
\begin{aligned}
\sum\limits_{k=1}^4 \sin(r \alpha_k \cdot x)=& 2 \sin\bigg(\frac{r(\alpha_1+\alpha_2) \cdot x}{2}\bigg)\cos\bigg(\frac{r(\alpha_1-\alpha_2)\cdot x}{2}\bigg) \\
&+2 \sin\bigg(\frac{r(\alpha_3+\alpha_4) \cdot x}{2}\bigg)\cos\bigg(\frac{r(\alpha_3-\alpha_4)\cdot x}{2}\bigg) \\
=& 2\sin\bigg(\frac{r(\alpha_1+\alpha_2) \cdot x}{2}\bigg)\bigg(\cos\bigg(\frac{r(\alpha_1-\alpha_2)\cdot x}{2}\bigg)-   \cos\bigg(\frac{r(\alpha_3-\alpha_4)\cdot x}{2}\bigg)     \bigg) \\
=& -2\sin(r \theta \cdot x)\bigg(\cos\bigg(\frac{r(\alpha_1-\alpha_2)\cdot x}{2}\bigg)-   \cos\bigg(\frac{r(\alpha_3-\alpha_4)\cdot x}{2}\bigg)     \bigg).\\
\end{aligned}
\end{equation*}
The square wave function $\psi(r \theta \cdot x)$ changes its sign at $r \theta \cdot x=k\pi, k \in \mathbb{Z}$, when $\sin(r \theta \cdot x)$ is equal to zero. Therefore the function $x \mapsto \psi(r \theta \cdot x)\sin(r \theta \cdot x)$ is continuous, and so is the term 
$\psi(r \theta \cdot x)\sum\limits_{k=1}^4 \sin(r \alpha_k \cdot x)$.

Similarly, using the formula $\cos(x)+\cos(y)=2\cos\left(\frac{x+y}{2} \right)\cos\left(\frac{x-y}{2} \right)$, we obtain 
\begin{align*}
\sum\limits_{k=1}^4 \cos(r \alpha_k \cdot x)=2\cos(r \theta \cdot x)\bigg(\cos\bigg(\frac{r(\alpha_1-\alpha_2)\cdot x}{2}\bigg)-   \cos\bigg(\frac{r(\alpha_3-\alpha_4)\cdot x}{2}\bigg)     \bigg).
\end{align*}
Then the continuity of $x \mapsto \psi\left({{r}\theta\cdot x+\frac{\pi}{2}}\right)\sum_{k=1}^4\cos\left({r}{\a_k \cdot x}\right)$ follows from the continuity of  $x \mapsto \psi\left({{r}\theta\cdot x+\frac{\pi}{2}}\right)\cos(r \theta \cdot x)$, and we finish the proof. 
\end{proof}

\subsubsection{Terminal condition}
\begin{proof}
Due to the continuity of $x \mapsto u^T(t,x)$  and the symmetry of $u^T$, we only need to show the equality $u^T(T,x)=\Phi(x)$ for the case $x^{(1)} <x^{(2)}<x^{(3)}< x^{(4)}$. Recall the definition of sine integral function $si(x)$ and cosine integral function $Ci(x)$ (see e.g. \cite{MR1225604}),
\begin{align*}
si(x)=-\int_x^{\infty} \frac{\sin(t)}{t} dt, \quad Ci(x)=-\int_x^{\infty} \frac{\cos(t)}{t} dt,
\end{align*}
and denote $$T_0=-\frac{\pi}{2 \theta \cdot x^o}, \quad A_k=\alpha_k \cdot x^o, \quad R_k=|A_kT_0|, \quad k=1, \dotso, 4. $$
Under the assumption $x^{(1)} <x^{(2)}<x^{(3)}< x^{(4)}$, it is easy to check the following inequalities 
\begin{align}\label{eq:range}
-\frac{3\pi}{2}<A_1T_0<-\frac{\pi}{2}<A_2T_0<A_3T_0<\frac{\pi}{2}<A_4 T_0 <\frac{3\pi}{2}.
\end{align}
According to \eqref{eq:defu}, we have 
\begin{align*}
2\sqrt{2} \sum\limits_{k=1}^4 x_k-8\sqrt{2}u^T(T,x)= \int_0^{\infty} \frac{\Lambda(r,x^o)}{r^2} dr. 
\end{align*}
Note that 
\[\psi\left({{r}\theta\cdot x^o+\frac{\pi}{2}}\right)=
\begin{cases}
-1, \quad r \in [(4n+1)T_0, (4n+3)T_0] \\
+1, \quad r \in [(4n-1)T_0, (4n+1)T_0], 
\end{cases}
\]
\[\psi\left({{r}\theta\cdot x^o}\right)=
\begin{cases}
-1, \quad r \in [4nT_0, (4n+2)T_0] \\
+1, \quad r \in [(4n+2)T_0, (4n+4)T_0].
\end{cases}
\]
We can rewrite the integral as infinite sum of integrals
\begin{align}\label{eq:integralpart}
\int_0^{\infty}   \frac{\Lambda(r,x^o)}{r^2} dr=&\int_0^{\infty}  \left( \psi\left({{r}\theta\cdot x^o+\frac{\pi}{2}}\right)\sum\limits_{k=1}^4 \frac{\cos(A_kr)}{r^2}-\frac{4}{r^2} \right) dr \notag \\
&-\int_0^{\infty} \psi\left({{r}\theta\cdot x^o}\right) \sum\limits_{k=1}^4 \frac{\sin(A_k4)}{r^2} dr \notag \\
=&\sum\limits_{n=0}^{\infty} (-1)^n \int_{2nT_0}^{(2n+2)T_0} \sum\limits_{k=1}^4\frac{\sin(A_k r)}{r^2}dr+\int_0^{T_0} \sum\limits_{k=1}^4\frac{\cos(A_kr)-1}{r^2} dr \notag \\
&+\sum\limits_{n=1}^{\infty} (-1)^n \int_{(2n-1)T_0}^{(2n+1)T_0} \sum\limits_{k=1}^4 \frac{\cos(A_kr)}{r^2} dr -\int_{T_0}^{\infty} \frac{4}{r^2} dr.
\end{align}
Our aim is to prove $\int_0^{\infty} \frac{\Lambda(r,x^o)}{r^2} dr=-4A_4$, which is equivalent to $u^T(T,x)=\Phi(x)$.

It is easy to check the following indefinite integral formulas,
\begin{align*}
& \int \frac{\sin(x)}{x^2} dx =Ci(x)-\frac{\sin(x)}{x}+\text{Constant}, \\
& \int \frac{\cos(x)}{x^2} dx=-si(x)-\frac{\cos(x)}{x}+\text{Constant}. \\
\end{align*}
Let us compute the integral 
\begin{align*}
\int_0^{2T_0} \frac{1}{r^2} \sum\limits_{k=1}^4 \sin (A_k r) dr =&  \sum\limits_{k=1}^4 A_k\left(Ci(2|A_k T_0|)-\frac{\sin(2|A_k T_0|)}{2|A_k T_0|}\right) \\ 
&-  \lim\limits_{\epsilon \to 0} \sum\limits_{k=1}^4 A_k \left(Ci(|A_k \epsilon|)-\frac{\sin(|A_k \epsilon|)}{|A_k \epsilon|} \right). \\
\end{align*}
According to $\sum\limits_{k=1}^4 A_k=0$ and $\lim\limits_{x \to 0} \frac{\sin(x)}{x}=1$, the term $\sum\limits_{k=1}^4 A_k \frac{\sin(|A_k \epsilon|    ) }{|A_k \epsilon|}$ vanishes. Since the expansion of $Ci(x)$ near $x=0$ is $\ln(x)+\gamma$, where $\gamma$ is the Euler-Mascheroni constant (see e.g. \cite{MR1411676}), we obtain that
$$\lim\limits_{\epsilon \to 0} \sum\limits_{k=1}^4 A_k Ci(|A_k \epsilon|)= \lim\limits_{\epsilon \to 0} \sum\limits_{k=1}^4 A_k(\ln(|A_k|)+\ln(\epsilon)+\gamma)=\sum\limits_{k=1}^4 A_k \ln(|A_k|) .$$
Accordingly, we have $$ \int_0^{2T_0} \frac{1}{r^2} \sum\limits_{k=1}^4 \sin (A_k r) dr =  \sum\limits_{k=1}^4 A_k\left(Ci(2|A_k T_0|)-\frac{\sin(2|A_k T_0|)}{2|A_k T_0|} )-\sum\limits_{k=1}^4 A_k \ln(|A_k|\right),$$
and similarly for each $n \in \mathbb{N}$, 
\begin{equation*}
\begin{aligned}
\int_{2nT_0}^{(2n+2)T_0}  \frac{1}{r^2} \sum\limits_{k=1}^4 \sin (A_k r) dr=& \sum\limits_{k=1}^4 A_k(Ci((2n+2)|A_k T_0|)-\frac{\sin((2n+2)|A_k T_0|)}{(2n+2)|A_k T_0|} ) \\
&-\sum\limits_{k=1}^4 A_k(Ci(2n|A_k T_0|)-\frac{\sin(2n|A_k T_0|)}{2n|A_k T_0|} ). \\
\end{aligned}
\end{equation*}
Therefore, we get the equation
\begin{align}\label{eq:sinsum}
\int_0^{+\infty} -\psi\left({{r}\theta\cdot x^o+\frac{\pi}{2}}\right) \sum\limits_{k=1}^4 \frac{\sin(A_kr)}{r^2} dr= &- \sum\limits_{k=1}^4 A_k\ln(|A_k|) + 2 \sum\limits_{k=1}^4  \sum\limits_{n=1}^{\infty} A_k (-1)^{n+1} Ci( 2n |A_k T_0|)  \notag  \\
& -  2 \sum\limits_{k=1}^4  \sum\limits_{n=1}^{\infty}  (-1)^{n+1} \frac{\sin(2nA_k T_0)}{2n T_0}.
\end{align}

Now we deal with the cosine term in \eqref{eq:integralpart}. Due to the equality $si(0)=-\frac{\pi}{2}$, it can be seen that 
\begin{align*}
\int_0^{T_0} \frac{1}{r^2} (\sum\limits_{k=1}^4 \cos (A_k r)-4) dr=&-\sum\limits_{k=1}^4 |A_k|si(|A_kT_0|)-\sum\limits_{k=1}^4 \frac{\cos(|A_kT_0|)}{T_0}+\frac{4}{T_0}   \\
&+\lim\limits_{\epsilon \to 0^+}  \left( |A_k|\sum\limits_{k=1}^4 si(|A_k \epsilon|)+\sum\limits_{k=1}^4 \frac{\cos(|A_k \epsilon|)-1}{\epsilon}\right) \\
=&-\frac{\pi}{2}\sum\limits_{k=1}^4 |A_k|-\sum\limits_{k=1}^4 |A_k|si(|A_kT_0|)-\sum\limits_{k=1}^4 \frac{\cos(|A_kT_0|)}{T_0}+\frac{4}{T_0}, \\
\end{align*}
and similarly 
\begin{equation*}
\begin{aligned}
\int_{(2n+1)T_0}^{(2n+3)T_0} \frac{1}{r^2} \sum\limits_{k=1}^4 \cos (A_k r) dr=&-\sum\limits_{k=1}^4 |A_k|si((2n+3)|A_kT_0|)-\sum\limits_{k=1}^4 \frac{\cos((2n+3)|A_kT_0|)}{(2n+3)T_0} \\
&+\sum\limits_{k=1}^4 |A_k|si((2n+1)|A_kT_0|)-\sum\limits_{k=1}^4 \frac{\cos((2n+1)|A_kT_0|)}{(2n+1)T_0}.
\end{aligned}
\end{equation*}
Then, in conjunction with the equality $\int_{T_0}^{+\infty} \frac{-4}{r^2} dr =-\frac{4}{T_0}$, we obtain that
\begin{align}\label{eq:cossum}
& \int_0^{+\infty}  \left( \psi\left({{r}\theta\cdot x^o}+\frac{\pi}{2} \right)   \sum\limits_{k=1}^4 \frac{\cos(A_kr)}{r^2} -\frac{4}{r^2} \right) dr \notag \\ 
& \quad \quad = -\frac{\pi}{2}\sum\limits_{k=1}^4|A_k| 
-  2 \sum\limits_{k=1}^4  \sum\limits_{n=1}^{\infty}  (-1)^{n+1} |A_k| si(| (2n-1)A_k T_0|) \notag \\
& \quad \quad \quad -  2 \sum\limits_{k=1}^4  \sum\limits_{n=1}^{\infty}  (-1)^{n+1} \frac{\cos((2n-1)A_k T_0)}{(2n-1) T_0}. 
\end{align}
Using the inverse Fourier transform, we have
\begin{align*}
&\sum\limits_{n=1}^{\infty}  (-1)^{n+1} \frac{\cos((2n-1)A_k T_0)}{(2n-1) T_0} =\frac{\pi}{4 T_0} sign \left(\tan\left(\frac{\pi}{4}+\frac{A_k T_0  }{2}       \right)\right), \\
& \sum\limits_{n=1}^{\infty}  (-1)^{n+1} \frac{\sin(2nA_k T_0)}{2n  T_0}=\frac{i\left(\log\left(1+e^{-i 2A_k T_0 }\right)-\log\left(1+e^{i 2A_k T_0 }\right) \right)   }{   4T_0 }.
\end{align*}
Recalling the inequalities \eqref{eq:range}, for $k=1,4$, we have $|A_kT_0| \in (\frac{\pi}{2}, \frac{3\pi}{2})$, and hence the term $sign \left(\tan\left(\frac{\pi}{4}+\frac{A_k T_0  }{2}       \right)\right)=-1$. For $k=2,3$, since $|A_kT_0| < \frac{\pi}{2}$, we get $sign \left(\tan\left(\frac{\pi}{4}+\frac{A_k T_0  }{2}       \right)\right)=1$, and therefore  
\begin{equation}\label{eq:arcsum}
\sum\limits_{k=1}^4\frac{\pi}{2T_0} sign \left(\tan\left(\frac{\pi}{4}+\frac{A_k T_0  }{2}       \right)\right)=0.
\end{equation}
 It can be seen that the function $$ x \mapsto i\left(\log(1+e^{-ix})-\log(1+e^{ix}) \right) \equiv i \ \log\left(\frac{1+e^{-ix}}{1+e^{ix}} \right) \equiv x \mod 2\pi $$
is $2\pi$-periodic, and equals to $x$ when restricted to $(-\pi, \pi)$. So that we obtain 
\[\sum\limits_{n=1}^{\infty}  (-1)^{n+1} \frac{\sin(2nA_k T_0)}{2n  T_0}=
\begin{cases}
\frac{2A_kT_0+2\pi}{4T_0},&  \quad \text{if} \ \ k=1, \\
\frac{2A_kT_0}{4T_0},&  \quad \text{if} \ \ k=2,3, \\
\frac{2A_kT_0-2\pi}{4T_0},&  \quad \text{if} \ \ k=4, \\
\end{cases}
\]
and hence 
\begin{equation}\label{eq:logsum}
\sum\limits_{k=1}^4\sum\limits_{n=1}^{\infty}  (-1)^{n+1} \frac{\sin(2nA_k T_0)}{2n  T_0}=0.
\end{equation}
Combining \eqref{eq:integralpart}, \eqref{eq:sinsum}, \eqref{eq:cossum}, \eqref{eq:arcsum} and \eqref{eq:logsum}, we simplify the expression,
\begin{align}
\int_0^{\infty} \frac{\Lambda(r,x^o) dr}{r^2}=&- \sum\limits_{k=1}^4 A_k\ln(|A_k|)-\frac{\pi}{2}\sum\limits_{k=1}^4|A_k| + 2 \sum\limits_{k=1}^4  \sum\limits_{n=1}^{\infty}  (-1)^{n+1} |A_k| si(| (2n-1)A_k T_0|) \notag \\
&-  2 \sum\limits_{k=1}^4  \sum\limits_{n=1}^{\infty}  (-1)^{n+1} \frac{\cos((2n-1)A_k T_0)}{(2n-1) T_0} 
+ 2 \sum\limits_{k=1}^4  \sum\limits_{n=1}^{\infty} A_k (-1)^{n+1} Ci( 2n |A_k T_0|) \notag \\
&-  2 \sum\limits_{k=1}^4  \sum\limits_{n=1}^{\infty}  (-1)^{n+1} \frac{\sin(2nA_k T_0)}{2n T_0} \notag \\
=&-\frac{\pi}{2}\sum\limits_{k=1}^4|A_k| + 2 \sum\limits_{k=1}^4  \sum\limits_{n=1}^{\infty}  (-1)^{n+1} |A_k| si(| (2n-1)A_k T_0|) \notag \\
&- \sum\limits_{k=1}^4 A_k\ln(|A_k|)+2 \sum\limits_{k=1}^4  \sum\limits_{n=1}^{\infty} A_k (-1)^{n+1} Ci( 2n |A_k T_0|). \label{eq:sincosintegral}
\end{align}

It remains to calculate the infinite sum including $Ci(x)$ and $si(x)$. Note that 
\begin{equation*}
-2\sum\limits_{n=1}^{\infty} (-1)^{n+1} si((2n-1)R_k)=2 \sum\limits_{n=1}^{\infty} \int_{(4n-3)R_k}^{(4n-1)R_k} \frac{\sin(r)}{r} dr.
\end{equation*}
and $\frac{\sin(z)}{z}=Im \frac{e^{iz}}{z}$for $z \in \mathbb{R}$.  We apply contour integral to $\frac{e^{iz}}{z}$. Denoting the curves in the counterclockwise direction by $$\gamma^k_n:=\{nR_ke^{i\theta}: \theta \in [0, \pi]                \}, $$
we have equalities $$2 \int_{(4n-3)R_k}^{(4n-1)R_k} -\int_{\gamma^k_{4n-3}} + \int_{\gamma^k_{4n-1}} \frac{e^{iz}}{z} dz =0, n \in \mathbb{N} .$$
Therefore, we obtain 
\begin{align}\label{equation:star}
-2\sum\limits_{n=1}^{\infty} (-1)^{n+1} si((2n-1)R_k)=&  \sum\limits_{n=1}^{\infty} (-1)^{n+1} Im \int_{\gamma^k_{2n-1}} \frac{e^{iz}}{z} dz   \notag \\
=&  \sum\limits_{n=1}^{\infty} (-1)^{n+1} Re \int_0^{\pi} e^{i(2n-1)R_ke^{i\theta}}d \theta. \tag{$*$}
\end{align}
According to the inequalities \eqref{eq:range}, we have $2R_k \not \in \{ -3\pi, -\pi, \pi, 3\pi\}$, and hence can exchange the infinite sum and the integral and compute the geometric series to obtain
\begin{equation}
-2\sum\limits_{n=1}^{\infty} (-1)^{n+1} si((2n-1)R_k)=Re \int_0^{\pi} \frac{e^{R_ke^{i\theta}}}{1+e^{2R_ke^{i\theta}}} d \theta .
\end{equation}

Similarly, we calculate $$2 \sum\limits_{n=1}^{\infty} (-1)^{n+1} Ci(2n R_k)=-2 Re\left( \sum\limits_{n=1}^{\infty} \int_{(4n-2)R_k}^{4nR_k} \frac{e^{iz}}{z} dz\right) .$$
Denoting the quarter of circles in the counterclockwise derivation by 
$$\tilde{\gamma}^k_n:=\{nR_ke^{i\theta}: \theta \in [0, \pi/2] \},$$ 
we obtain that 
\begin{align*}
0=& \int_{(4n-2)R_k}^{4nR_k}+\int_{\tilde{\gamma}^k_{4n}}-\int_{i(4n-2)R_k}^{i4nR_k}-\int_{\tilde{\gamma}^k_{4n-2}} \frac{e^{iz}}{z} dz \\
=& \int_{(4n-2)R_k}^{4nR_k} \frac{e^{iz}}{z}dz + i\int_{\frac{\pi}{2}}^0 e^{(4n-2)R_ke^{i\theta}}d \theta + i\int_0^{\frac{\pi}{2}} e^{4nR_ke^{i\theta}} d \theta + \int_{4nR_k}^{(4n-2)R_k}\frac{e^{-r}}{r}dr.
\end{align*}
Recalling the definition of integral exponential function for $x >0$, $$E_1(x)=\int_x^{+\infty} \frac{e^{-r}}{r} dr=\int_0^{+\infty} \exp(-xe^t)dt,$$
it can be seen that 
\begin{align*}
-2Re \int_{(4n-2)R_k}^{4nR_k} \frac{e^{iz}}{z}dz=&2Re \left(-i\left(\int_0^{\frac{\pi}{2}} e^{(4n-2)R_ke^{i\theta}}d \theta- \int_0^{\frac{\pi}{2}} e^{4nR_ke^{i\theta}} d \theta\right) \right) \\
&+2(E_1(4nR_k)-E_1((4n-2)R_k))  \\
=&2Im \left(\int_0^{\frac{\pi}{2}} e^{(4n-2)R_ke^{i\theta}}d \theta- \int_0^{\frac{\pi}{2}} e^{4nR_ke^{i\theta}} d \theta\right) \\
&+ 2(E_1(4nR_k)-E_1((4n-2)R_k)) .
\end{align*}
By direct computation, we have 
\begin{align}\label{equation:starstar}
2 \sum\limits_{n=1}^{\infty} (-1)^{n+1} Ci(2n R_k)&=2Im\left(\int_0^{\frac{\pi}{2}} \frac{e^{i2R_ke^{i\theta}}}{1+ e^{i2R_ke^{i\theta}}} d \theta\right) + 2 \left(\sum\limits_{n=1}^{\infty}(E_1(4nR_k)-E_1((4n-2)R_k)) \right)  \notag \\
&=2Im\left(\int_0^{\frac{\pi}{2}} \frac{e^{i2R_ke^{i\theta}}}{1+ e^{i2R_ke^{i\theta}}} d \theta\right)-2 \int_0^{\infty} \frac{e^{-2R_ke^r}}{1+e^{-2R_ke^r}}dr  \notag \\
&=2Im\left(\int_0^{\frac{\pi}{2}} \frac{e^{i2R_ke^{i\theta}}}{1+ e^{i2R_ke^{i\theta}}} d \theta\right)-2\int_{R_k}^{\infty} \frac{e^{-2t}}{t(1+e^{-2t})} dt, \tag{$**$} 
\end{align}
where the last equation follows from the change of variable $t=R_ke^r$. 

Then, we can further simplify the expression \eqref{eq:sincosintegral} using \eqref{equation:star} and \eqref{equation:starstar},
\begin{align}\label{sl1}
\int_0^{\infty} \frac{\Lambda(r,x^o)}{r^2}dr=& - \sum\limits_{k=1}^4 A_k\ln(|A_k|)-\frac{\pi}{2}\sum\limits_{k=1}^4|A_k|   + \sum\limits_{k=1}^4 |A_k| Re \int_0^{\pi} \frac{e^{iR_ke^{i\theta}}}{1+e^{i2R_ke^{i\theta}}}d \theta \notag \\
&+\sum\limits_{k=1}^4 2A_k Im\int_0^{\frac{\pi}{2}} \frac{e^{i2R_ke^{i\theta}}}{1+ e^{i2R_ke^{i\theta}}} d \theta-2 A_k \int_{R_k}^{\infty} \frac{e^{-2t}}{t(1+e^{-2t})} dt .
\end{align}
Let us denote
\begin{equation*}
\begin{aligned}
I_k& =\int_0^{\pi} \frac{e^{iR_ke^{i\theta}}}{1+e^{i2R_ke^{i\theta}}}d \theta=-i\int_{\gamma^k_1} \frac{e^{iz}}{z(1+e^{i2z})}dz, \\
J_k& = \int_0^{\frac{\pi}{2}} \frac{e^{i2R_ke^{i\theta}}}{1+ e^{i2R_ke^{i\theta}}} d \theta= -i \int_{\tilde{\gamma}^k_1} \frac{e^{i2z}}{z(1+e^{i2z})}dz. \\
\end{aligned}
\end{equation*}

\begin{figure}
\begin{tikzpicture}
\def\bigradius{2}
\def\littleradius{0.2}

\draw [help lines,->] (-1.5*\bigradius, 0) -- (1.5*\bigradius,0);
 \draw [help lines,->]   (0, -0.2*\bigradius) -- (0, 1.2*\bigradius);
\draw[black, thick,   decoration={ markings,
      mark=at position 0.35 with {\arrow{latex}},
      mark=at position 0.73 with {\arrow{latex}},  
      mark=at position 0.955 with {\arrow{latex}}}, 
      postaction={decorate}]  
   (0:\bigradius) arc (0:180:\bigradius)
   --(180:\littleradius) arc (180:0:\littleradius)
   -- cycle;
  
\node at (-1.5,2.1) {$I_k,k=2,3$};
\node at (2.7,-0.2) {$\frac{\pi}{2}$};
\end{tikzpicture}
\begin{tikzpicture}
\def\bigradius{2}
\def\littleradius{0.2}

\draw [help lines,->] (-1.5*\bigradius, 0) -- (1.5*\bigradius,0);
 \draw [help lines,->]   (0, -0.2*\bigradius) -- (0, 1.2*\bigradius);
\draw[black, thick,   decoration={ markings,
      mark=at position 0.35 with {\arrow{latex}},
      mark=at position 0.73 with {\arrow{latex}},  
      mark=at position 0.87 with {\arrow{latex}}}, 
      postaction={decorate}]  
   (0:\bigradius) arc (0:180:\bigradius)--(180:1.5) arc (180:0:\littleradius)
     --(180:\littleradius) arc (180:0:\littleradius)--(0:1.1) arc (180:0:\littleradius)
     --cycle;
  
\node at (-1.5,2.1) {$I_k, k=1,4$};
\node at (1.3,-0.2) {$\frac{\pi}{2}$};
\end{tikzpicture}
\caption{Contour of $I_k$}
\label{figure1}
\end{figure}
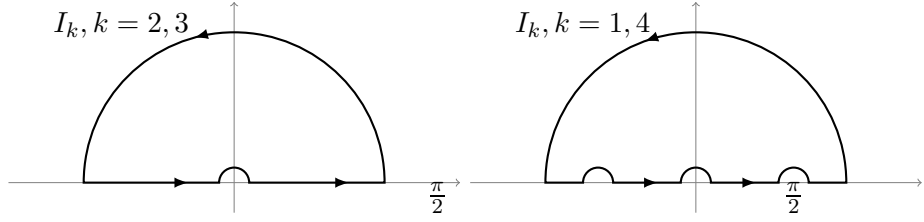
For $k=2,3$, we have $R_k <\frac{\pi}{2}$, and therefore $0$ is the only pole of complex function $ z \mapsto \frac{e^{iz}}{z(1+e^{i2z})}$ over the interval $[-R_k,R_k]$. According to the contour integral (see Figure \ref{figure1}), we have that 
\begin{equation*}
\begin{aligned}
0=&\int_{\gamma^k_1}+\int_{-R_k}^{-\epsilon}-\int_{\{\epsilon e^{i\theta}: \theta \in [0, \pi]    \}}+\int_{\epsilon}^{R_k}\frac{-i e^{iz}}{z(1+e^{i2z})}dz \\
=&I_k-i \int_{-R_k}^{-\epsilon} \frac{e^{it}}{t(1+e^{i2t})} dt -\int_0^{\pi} \frac{e^{i\epsilon e^{i\theta}}}{1+e^{i2\epsilon e^{i\theta}}}d \theta-i \int_{\epsilon}^{R_k} \frac{e^{it}}{t(1+e^{i2t})}dt.
\end{aligned}
\end{equation*}
Since $\frac{e^{it}}{t(1+e^{i2t})}=\frac{e^{it}+e^{-it}}{t|1+e^{i2t}|^2}$ is real, and $\lim\limits_{\epsilon \to 0} \frac{e^{i\epsilon e^{i\theta}}}{1+e^{i2\epsilon e^{i\theta}}}=\frac{1}{2}$, we obtain that $Re I_k= \frac{\pi}{2}.$
For $k=1,4$, since $R_k \in (\frac{\pi}{2}, \frac{3\pi}{2})$, we have $-\frac{\pi}{2},0,\frac{\pi}{2}$ are three poles of the complex function $ z \mapsto \frac{e^{iz}}{z(1+e^{i2z})}$ over the interval $[-R_k,R_k]$. Again by contour integral (see Figure \ref{figure1}) the real part of $I_k$ is equal to the integral around the three poles, 
\begin{align*}
ReI_k & =\lim\limits_{\epsilon \to 0} Re\int_0^{\pi} \frac{i e^{i\epsilon e^{i \theta}}\epsilon e^{i \theta}}{1-e^{i2\epsilon e^{i \theta}}}\bigg(\frac{1}{\frac{\pi}{2}+\epsilon e^{i \theta}}+\frac{1}{\frac{\pi}{2}-\epsilon e^{i \theta}}\bigg) d \theta +\frac{\pi}{2}  \notag \\
&=\int_0^{\pi} \lim\limits_{\epsilon \to 0} Re  \frac{i e^{i\epsilon e^{i \theta}}\epsilon e^{i \theta}}{1-e^{i2\epsilon e^{i \theta}}}\bigg(\frac{1}{\frac{\pi}{2}+\epsilon e^{i \theta}}+\frac{1}{\frac{\pi}{2}-\epsilon e^{i \theta}}\bigg) d \theta +\frac{\pi}{2}  \notag \\
&=\int_0^{\pi} -\frac{2}{\pi} d \theta +\frac{\pi}{2}=\frac{\pi}{2}-2, 
\end{align*}
and hence 
\begin{align}\label{Re}
\sum\limits_{k=1}^4|A_k| Re I_k=\frac{\pi}{2} \sum\limits_{k=1}^4 |A_k| - 2|A_1|-2|A_4|.
\end{align}
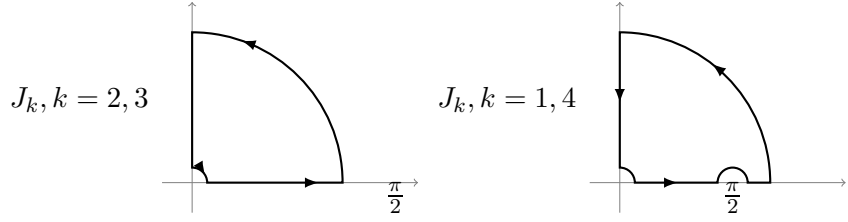
\begin{figure}
\begin{tikzpicture}
\def\bigradius{2}
\def\littleradius{0.2}

\draw [help lines,->] (-0.2*\bigradius, 0) -- (1.5*\bigradius,0);
 \draw [help lines,->]   (0, -0.2*\bigradius) -- (0, 1.2*\bigradius);
\draw[black, thick,   decoration={ markings,
      mark=at position 0.35 with {\arrow{latex}},
      mark=at position 0.73 with {\arrow{latex}},  
      mark=at position 0.955 with {\arrow{latex}}}, 
      postaction={decorate}]  
   (0:\bigradius) arc (0:90:\bigradius)--(90:\littleradius) arc (90:0:\littleradius)
   -- cycle;
  
\node at (-1.5,1.1) {$J_k,k=2,3$};
\node at (2.7,-0.2) {$\frac{\pi}{2}$};
\end{tikzpicture}
\begin{tikzpicture}
\def\bigradius{2}
\def\littleradius{0.2}

\draw [help lines,->] (-0.2*\bigradius, 0) -- (1.5*\bigradius,0);
 \draw [help lines,->]   (0, -0.2*\bigradius) -- (0, 1.2*\bigradius);
\draw[black, thick,   decoration={ markings,
      mark=at position 0.25 with {\arrow{latex}},
      mark=at position 0.56 with {\arrow{latex}},  
      mark=at position 0.80with {\arrow{latex}}}, 
      postaction={decorate}]  
   (0:\bigradius) arc (0:90:\bigradius)--(90:\littleradius) arc (90:0:\littleradius)
   --(0:1.3) arc (180:0:\littleradius)
   -- cycle;
  
\node at (-1.5,1.1) {$J_k,k=1,4$};
\node at (1.5,-0.2) {$\frac{\pi}{2}$};
\end{tikzpicture}
\caption{Contour of $J_k$}
\label{figure2}
\end{figure}
For $k=2, 3$, we apply contour integral to $J_k$ (see Figure \ref{figure2}), 
\begin{equation*}
\begin{aligned}
0&=\int_{\tilde{\gamma}^k_1}+\int_{iR_k}^{i\epsilon}-\int_{\{\epsilon e^{i\theta}: \theta \in [0, \pi/2]    \}}+\int_{\epsilon}^{R_k}\frac{-i e^{i2z}}{z(1+e^{i2z})}dz \\
&=J_k-i \int_{R_k}^{\epsilon} \frac{e^{-2t}}{t(1+e^{-2t})} dt +\int_{\frac{\pi}{2}}^0  \frac{e^{i2 \epsilon e^{i\theta}}}{1+ e^{i2 \epsilon e^{i\theta}}} d \theta-i \int_{\epsilon}^{R_k} \frac{e^{i2t}}{t(1+e^{i2t})} dt.
\end{aligned}
\end{equation*}
Noting that imaginary part of $i\frac{e^{i2t} }{t(1+e^{i2t})}$ is just $\frac{1}{2t} $, we obtain that $$Im J_k=\int_{R_k}^{\epsilon} \frac{e^{-2t}}{t(1+e^{-2t})} dt + Im \int_0^{\frac{\pi}{2}}  \frac{e^{i2 \epsilon e^{i\theta}}}{1+ e^{i2 \epsilon e^{i\theta}}} d \theta+\int_{\epsilon}^{R_k} \frac{1}{2t} dt.$$
For $k=1,4$, $z=\frac{\pi}{2}$ is the other pole over the interval $[0, R_k]$ (see Figure \ref{figure2}), and we have 
\begin{equation*}
\begin{aligned}
0=& J_k-i \int_{R_k}^{\epsilon} \frac{e^{-2t}}{t(1+e^{-2t})} dt +\int_{\frac{\pi}{2}}^0  \frac{e^{i2 \epsilon e^{i\theta}}}{1+ e^{i2 \epsilon e^{i\theta}}} d \theta-i \int_{\epsilon}^{\frac{\pi}{2}-\epsilon} \frac{e^{i2t}}{t(1+e^{i2t})} dt \\
&-i \int_{\frac{\pi}{2}+\epsilon}^{R_k} \frac{e^{i2t}}{t(1+e^{i2t})} dt + \int_0^{\pi} \frac{e^{i2\epsilon e^{i\theta}} \epsilon e^{i \theta}}{(\frac{\pi}{2}+ \epsilon e^{i \theta})(1-e^{i2\epsilon e^{i\theta}})} d \theta,
\end{aligned}
\end{equation*}
and therefore, 
\begin{equation*}
\begin{aligned}
Im J_k=&\int_{R_k}^{\epsilon} \frac{e^{-2t}}{t(1+e^{-2t})} dt + Im \int_0^{\frac{\pi}{2}}  \frac{e^{i2 \epsilon e^{i\theta}}}{1+ e^{i2 \epsilon e^{i\theta}}} d \theta+\int_{\epsilon}^{R_k} \frac{1}{2t} dt \\
&-\int_{\frac{\pi}{2}-\epsilon}^{\frac{\pi}{2}+\epsilon} \frac{1}{2t} dt- Im \int_0^{\pi} \frac{e^{i2\epsilon e^{i\theta}} \epsilon e^{i \theta}}{(\frac{\pi}{2}+ \epsilon e^{i \theta})(1-e^{i2\epsilon e^{i\theta}})} d \theta.
\end{aligned}
\end{equation*}
Take 
\begin{align*}
P^1_{\epsilon}=Im \int_0^{\frac{\pi}{2}}  \frac{e^{i2 \epsilon e^{i\theta}}}{1+ e^{i2 \epsilon e^{i\theta}}} d \theta, \quad \quad 
P^2_{\epsilon}=Im \int_0^{\pi} \frac{e^{i2\epsilon e^{i\theta}} \epsilon e^{i \theta}}{(\frac{\pi}{2}+ \epsilon e^{i \theta})(1-e^{i2\epsilon e^{i\theta}})} d \theta.
\end{align*}
Then we compute 
\begin{align*}
2Im J_k-2\int_{R_k}^{\infty} \frac{e^{-2t}}{t(1+e^{-2t})} dt - \ln(|A_k|)=&-2\int_{\epsilon}^{\infty} \frac{e^{-2t}}{t(1+e^{-2t})} dt+(\ln(R_k)-\ln(|A_k|)-\ln(\epsilon) ) \\
&+2P^1_{\epsilon}- \mathbbm{1}_{\{k=1,4\}}\left(\ln \left(\frac{\pi/2+\epsilon}{\pi/2-\epsilon} \right)+2P_{\epsilon}^2 \right).
\end{align*}
As a result of $\sum\limits_{k=1}^4 A_k=0$, $ \ln(R_k)-\ln(|A_k|)=\ln(T_0)$ and $\lim\limits_{\epsilon \to 0} P^2_{\epsilon}=1 $, it can be seen that 
\begin{align}\label{Im}
\sum\limits_{k=1}^4  A_k\bigg( 2ImJ_k- & \ln(|A_k|)-2\int_{R_k}^{\infty} \frac{e^{-2t}}{t(1+e^{-2t})} dt  \bigg) \notag \\
= & - 2\sum_{k=1}^4 A_k \int_{\epsilon}^{\infty} \frac{e^{-2t}}{t(1+e^{-2t})} dt +\sum\limits_{k=1}^4 A_k(  \ln(T_0)-\ln(\epsilon) ) \notag  \\
&+\sum\limits_{k=1}^4  A_k P^1_{\epsilon}-(A_1+A_4)\left(\ln \left(\frac{\pi/2+\epsilon}{\pi/2-\epsilon} \right)+2P_{\epsilon}^2 \right) \notag \\
=&- \lim\limits_{\epsilon \to 0} (A_1+A_4)\left(\ln \left(\frac{\pi/2+\epsilon}{\pi/2-\epsilon} \right)+2P_{\epsilon}^2 \right)  \notag \\
=&-2A_1-2A_4. 
\end{align}

Combining \eqref{sl1}, \eqref{Re}, \eqref{Im}, we obtain 

\begin{align}\label{sl2}
\int_0^{\infty}\frac{\Lambda(r,x^o)}{r^2} dr=-2|A_1|-2|A_4|-2A_1-2A_4=-4A_4,
\end{align}
which concludes the result. 

\end{proof}

\subsubsection{Smoothness} 

\begin{proof}
{\it Step 1: Equation \eqref{1st}.}
As a result of  $$\sum\limits_{k=1}^4 \sin(r\a_k \cdot x^{o})=\sum\limits_{k=1}^4 \sin(r\a_k \cdot x), \sum\limits_{k=1}^4 \cos(r\a_k \cdot x^{o})=\sum\limits_{k=1}^4 \cos(r\a_k \cdot x),$$
we obtain that 
\begin{align*}
u^T(t,x)=& \frac{-1}{16\sqrt{2}}\int_{-\infty}^{\infty} \frac{e^{-(T-t)r^2}}{r^2}\left( \psi\left({{r}\theta\cdot x^{o}+\frac{\pi}{2}}\right)\sum\limits_{k=1}^4 \cos(r\a_k \cdot x)-4  \right. \\
& \left. - \psi\left( r \theta\cdot x^{o}\right)\sum\limits_{k=1}^4 \sin(r\a_k \cdot x) \right)dr +\frac{1}{4}\sum_{k=1}^4x_k+\frac{1}{2} \sqrt{\frac{(T-t)\pi}{2} }.
\end{align*}
To stress the dependence of $T_0$ on $x$, we denote it as $T_0(x):=-\frac{\pi}{2\theta \cdot x^{o}}.$ Since $\theta \cdot x^{o} < 0$, define two functional series $F_l(x), G_l(x), l \in \mathbb{Z}$ by
\begin{align*}
F_l(x):=\int_{(2l-1)T_0(x)}^{(2l+1)T_0(x)} \frac{e^{-(T-t)r^2}}{r^2}\left( \psi\left({{r}\theta\cdot x^{o}+\frac{\pi}{2}}\right)\sum\limits_{k=1}^4 \cos(r\a_k \cdot x)-4  \right) dr, 
\end{align*}
\begin{align*}
G_l(x):=-\int_{2lT_0(x)}^{(2l+2)T_0(x)} \frac{e^{-(T-t)r^2}}{r^2} \left(\psi\left( r \theta\cdot x^{o}\right)\sum\limits_{k=1}^4 \sin(r\a_k \cdot x) \right) dr. 
\end{align*}
Then we have $$u^T(t,x)=-\frac{1}{16 \sqrt{2}} \sum\limits_{l \in \mathbb{Z}}(F_l(x)+G_l(x))+\frac{1}{4}\sum_{k=1}^4x_k+\frac{1}{2} \sqrt{\frac{(T-t)\pi}{2} }.$$
Noting that $\sum\limits_{k=1}^4 \cos(r \a_k \cdot x)=0$ at endpoints $r=(2l-1)T_0(x), l \in \mathbb{Z}$, the partial derivative of $F_l(x)$ is given by 
\begin{align*}
\pa^{+}_{x_i}F_l(x)=&-\int_{(2l-1)T_0(x)}^{(2l+1)T_0(x)} \frac{e^{-(T-t)r^2}}{r} \sum\limits_{k=1}^4 \a_{k,i}\psi\left( r \theta\cdot x^{o}+\frac{\pi}{2} \right) \sin(r \a_k \cdot x) dr \\
&-4(2l+1) \pa^{+}_{x_i}T_0(x)\frac{e^{-(T-t)(2l+1)^2T_0^2(x)}}{(2l+1)^2T_0^2(x)}+4(2l-1) \pa^{+}_{x_i}T_0(x)\frac{e^{-(T-t)(2l-1)^2T_0^2(x)}}{(2l-1)^2T_0^2(x)}, \\
\end{align*}
and similarly  
\begin{align*}
\pa^{+}_{x_i}G_l(x)=-\int_{2lT_0(x)}^{(2l+2)T_0(x)} \frac{e^{-(T-t)r^2}}{r} \sum\limits_{k=1}^4 \a_{k,i}\psi\left( r \theta\cdot x^{o} \right) \cos(r \a_k \cdot x) dr.
\end{align*}
It is well-known that summation and differentiation are interchangeable if the partial sum of derivatives converges uniformly. Since $\sum\limits_{l \in \mathbb{Z}} \pa^{+}_{x_i}F_l(x)$ and $\sum\limits_{l \in \mathbb{Z}} \pa^{+}_{x_i}G_l(x)$ converge uniformly in any bounded region of $x$, we conclude that,
\begin{align*}
\pa^{+}_{x_i}u^T(t,x)=&-\frac{1}{16 \sqrt{2}} \sum\limits_{l \in \mathbb{Z}} \left(\pa_{x_i}F_l(x)+\pa_{x_i}G_l(x)     \right) +\frac{1}{4} \notag \\
=&\frac{1}{16\sqrt{2} }\int_{-\infty}^{\infty} \frac{e^{-(T-t)r^2}}{r}\sum_{k=1}^4\a_{k,i}\left(\psi\left({{r}\theta\cdot x^{o}+\frac{\pi}{2}}\right)\sin\left({r}{\a_k \cdot x}\right) \right. \notag \\
 &\left. + \psi\left({{r}\theta\cdot x^{o}}\right)\cos\left({r}{\a_k \cdot x}\right)\right)dr+\frac{1}{4}. \notag \\
\end{align*}
We can calculate $\pa^{-}_{x_i}u^T(t,x)$ in the exactly same way, and find that it has the same expression with $\pa^{+}_{x_i}u^T(t,x)$. Therefore we proved the result \eqref{1st}.

{\it Step 2: Equation \eqref{1zero}.}
If $\theta \cdot x^{o}=0$, then all the coordinates of $x$ are equal, i.e., $x=(x_1,x_1,x_1,x_1)$. Let us compute the derivative of $u^T(t,x)$ by definition. Take  $\epsilon >0$ and denote $\big(x_1+\epsilon,x_1,x_1,x_1\big)$ simply by $x+\epsilon$. Then we have 
\begin{align*}
u^T(t,x+\epsilon)=& \frac{-1}{8\sqrt{2}}\int_{0}^{\infty} \frac{e^{-(T-t)r^2}}{r^2}\left( \psi\left({\frac{\epsilon r}{\sqrt{2}}+\frac{\pi}{2}}\right) \left( \cos\left(\frac{3\epsilon r}{\sqrt{2}}\right) +3\cos\left(\frac{3\epsilon r}{\sqrt{2}}\right) \right) -4  \right. \notag \\
& \left. - \psi\left( \frac{\epsilon r}{\sqrt{2}}\right) \left( \sin\left(\frac{3\epsilon r}{\sqrt{2}}\right)-3\sin\left(\frac{\epsilon r}{\sqrt{2}}\right)   \right) \right) dr +x_1+\frac{1}{4}\epsilon+\frac{1}{2} \sqrt{\frac{(T-t)\pi}{2} }.
\end{align*}
In order to conclude our result, it remains to show that 
\begin{align*}
0=& \lim\limits_{\epsilon \to 0} \frac{1}{\epsilon} \int_{0}^{\infty} \frac{e^{-(T-t)r^2}}{r^2}\left( \psi\left({\frac{\epsilon r}{\sqrt{2}}+\frac{\pi}{2}}\right) \left( \cos\left(\frac{3\epsilon r}{\sqrt{2}}\right) +3\cos\left(\frac{\epsilon r}{\sqrt{2}}\right) \right) -4  \right. \notag \\
& \left. - \psi\left( \frac{\epsilon r}{\sqrt{2}}\right) \left( \sin\left(\frac{3\epsilon r}{\sqrt{2}}\right)-3\sin\left(\frac{\epsilon r}{\sqrt{2}}\right)   \right) \right) dr.
\end{align*}
According to the estimation  $$\int_{\frac{1}{\sqrt{\epsilon}}}^{\infty} \frac{e^{-(T-t)r^2}}{r^2}dr \leq \epsilon \int_{\frac{1}{\sqrt{\epsilon}}}^{\infty} e^{-(T-t)r^2} dr,  $$ 
it can be seen that 
\begin{align*}
0=&\lim\limits_{\epsilon \to 0} \int_{\frac{1}{\sqrt{\epsilon}}}^{\infty} 12e^{-(T-t)r^2} dr \geq \lim\limits_{\epsilon \to 0} \frac{1}{\epsilon} \int_{\frac{1}{\sqrt{\epsilon}}}^{\infty} \frac{12e^{-(T-t)r^2}}{r^2}dr \\
\geq & \lim\limits_{\epsilon \to 0} \frac{1}{\epsilon} \int_{\frac{1}{\sqrt{\epsilon}}}^{\infty} \frac{e^{-(T-t)r^2}}{r^2}\left| \left( \psi\left({\frac{\epsilon r}{\sqrt{2}}+\frac{\pi}{2}}\right) \left( \cos\left(\frac{3\epsilon r}{\sqrt{2}}\right) +3\cos\left(\frac{\epsilon r}{\sqrt{2}}\right) \right) -4  \right. \right. \notag \\
& \left.\left. - \psi\left( \frac{\epsilon r}{\sqrt{2}}\right) \left( \sin\left(\frac{3\epsilon r}{\sqrt{2}}\right)-3\sin\left(\frac{\epsilon r}{\sqrt{2}}\right)   \right) \right) \right| dr.
\end{align*}
Now, both $\psi\left({\frac{\epsilon r}{\sqrt{2}}+\frac{\pi}{2}}\right)$ and $\psi\left( \frac{\epsilon r}{\sqrt{2}}\right)$ are positive over the interval $\left[0, \frac{1}{\sqrt{\epsilon}} \right]$. In conjunction with two equalities 
$$\cos\left(\frac{3\epsilon r}{\sqrt{2}}\right) +3\cos\left(\frac{\epsilon r}{\sqrt{2}}\right)=4 \cos^3 \left(\frac{\epsilon r}{\sqrt{2}} \right),$$
$$\sin\left(\frac{3\epsilon r}{\sqrt{2}}\right)-3\sin\left(\frac{\epsilon r}{\sqrt{2}}\right)=-4\sin^3 \left( \frac{\epsilon r}{\sqrt{2}}  \right), $$
we make the estimation 
\begin{align*}
  \bigg| \frac{1}{\epsilon} \int_0^{\frac{1}{\sqrt{\epsilon}}} \frac{e^{-(T-t)r^2}}{r^2} \Lambda(t, x+\epsilon) dr  \bigg| 
  & \leq  \bigg|   \frac{1}{\epsilon}  \int_0^{\frac{1}{\sqrt{\epsilon}}} \frac{e^{-(T-t)r^2}}{r^2}\left(4 \cos^3 \left(\frac{\epsilon r}{\sqrt{2}} \right)-4+ 4\sin^3 \left( \frac{\epsilon r}{\sqrt{2}}  \right)  \right) dr \bigg| \\
  & \leq \epsilon   \int_0^{\frac{1}{\sqrt{\epsilon}}} {e^{-(T-t)r^2}} \left( \frac{ 4-4 \cos^3 \left(\frac{\epsilon r}{\sqrt{2}} \right)      }{(\epsilon r)^2}+\frac{4\sin^3 \left( \frac{\epsilon r}{\sqrt{2}}  \right) }{(\epsilon r)^2} \right) dr .
\end{align*}
Since the integral $\int_0^{\frac{1}{\sqrt{\epsilon}}} {e^{-(T-t)r^2}} \left( \frac{ 4-4 \cos^3 \left(\frac{\epsilon r}{\sqrt{2}} \right)      }{(\epsilon r)^2}+\frac{4\sin^3 \left( \frac{\epsilon r}{\sqrt{2}}  \right) }{(\epsilon r)^2} \right) dr $ is bounded as $\epsilon \to 0$, we conclude that $$\lim\limits_{\epsilon \to 0} \frac{1}{\epsilon} \int_0^{\infty} \frac{e^{-(T-t)r^2}}{r^2} \Lambda(t, x+\epsilon) dr =0, $$
and hence $\pa_{x_1}u^T(t,x)=\frac{1}{4}$. Similarly, we can prove that  $\pa_{x_i}u^T(t,x)=\frac{1}{4}, i=2,3,4.$

{\it Step 3: Equation \eqref{2nd}.}
Define two functional series $H_l(x), K_l(x), l \in \mathbb{Z}$ by
\begin{align*}
H_l(x):=&\int_{(2l-1)T_0(x)}^{(2l+1)T_0(x)} \frac{e^{-(T-t)r^2}}{r}\sum_{k=1}^4\a_{k,i}\psi\left({{r}\theta\cdot x^{o}+\frac{\pi}{2}}\right)\sin\left({r}{\a_k \cdot x}\right)  dr,
\end{align*}
\begin{align*}
K_l(x):=\int_{2lT_0(x)}^{(2l+2)T_0(x)} \frac{e^{-(T-t)r^2}}{r}\sum_{k=1}^4\a_{k,i} \psi\left({{r}\theta\cdot x^{o}}\right)\cos\left({r}{\a_k \cdot x}\right)dr.
\end{align*}
Then we have $\pa_{x_i}u^T(t,x)=\frac{1}{4}+\frac{1}{16\sqrt{2}}\sum\limits_{l \in \mathbb{Z}} (H_l(x)+K_l(x))$. We compute the right-hand derivatives of $H_l(x), K_l(x)$, 
\begin{align}\label{le1}
\pa^{+}_{x_j} H_l(x)=&\int_{(2l-1)T_0(x)}^{(2l+1)T_0(x)} e^{-(T-t)r^2}\sum_{k=1}^4\a_{k,i} \a_{k,j} \psi\left({{r}\theta\cdot x^{o}+\frac{\pi}{2}}\right)\cos\left({r}{\a_k \cdot x}\right)  dr \notag \\
&+ \pa^{+}_{x_j} (\theta \cdot x^{o}) \frac{(-1)^le^{- \frac{(T-t)(\pi (l+1/2))^2}{(\theta\cdot x^{o})^2}}}{\theta\cdot x^{o}}\sum_{k=1}^4 2 \a_{k,i}\sin \left(\frac{\a_k \cdot x \pi(l+1/2)}{\theta\cdot x^{o}}\right)  \notag \\
&- \pa^{+}_{x_j} (\theta \cdot x^{o}) \frac{(-1)^le^{- \frac{(T-t)(\pi (l-1/2))^2}{(\theta\cdot x^{o})^2}}}{\theta\cdot x^{o}}\sum_{k=1}^4 2 \a_{k,i}\sin \left(\frac{\a_k \cdot x \pi(l-1/2)}{\theta\cdot x^{o}}\right) ,
\end{align}
\begin{align}\label{le2}
\pa^{+}_{x_j} K_l(x)=&-\int_{2lT_0(x)}^{(2l+2)T_0(x)} e^{-(T-t)r^2} \sum_{k=1}^4\a_{k,i}\a_{k,j} \psi\left({{r}\theta\cdot x^{o}}\right)\sin\left({r}{\a_k \cdot x}\right)dr \notag \\
&+ \pa^{+}_{x_j} (\theta \cdot x^{o})\frac{(-1)^{l}e^{- \frac{(T-t)(\pi (l+1))^2}{(\theta\cdot x^{o})^2}}}{\theta \cdot x^{o}}\sum_{k=1}^4 2\a_{k,i}\cos \left(\frac{\a_k \cdot x \pi (l+1)}{\theta\cdot x^{o}}\right) \notag \\
&-\pa^{+}_{x_j} (\theta \cdot x^{o})\frac{(-1)^{l}e^{- \frac{(T-t)(\pi l)^2}{(\theta\cdot x^{o})^2}}}{\theta \cdot x^{o}}\sum_{k=1}^4 2\a_{k,i}\cos \left(\frac{\a_k \cdot x \pi l}{\theta\cdot x^{o}}\right).
\end{align}
Replacing all the $\pa^{+}_{x_j}$ with $\pa^{-}_{x_j}$, we obtain the left hand side derivatives of $H_l(x)$ and $K_l(x)$. It can be easily checked that if $\theta \cdot x^{o} < 0, x^{(2)}<x^{(3)}$, the function $x \mapsto \theta \cdot x^{o}$ is differentiable, and hence $\pa^{+}_{x_j} H_l(x)=\pa^{-}_{x_j} H_l(x)$, $ \pa^{+}_{x_j} K_l(x)=\pa^{-}_{x_j}K_l(x)$. 
Since $\sum\limits_{l \in \mathbb{Z}}\pa_{x_j} H_j(x)$ and $\sum\limits_{l \in \mathbb{Z}}\pa_{x_j} H_j(x)$ converge uniformly in any bounded region of $x$, we can interchange summation and differentiation and obtain \eqref{2nd}.

{\it Step 4: Equation \eqref{3rd}.}
If $\theta \cdot x^{o}<0, x^{(2)}=x^{(3)}$, the right derivative $\pa^{+}_{x_j} (\theta \cdot x^{o})$ may not equal to the left derivative $\pa^{-}_{x_j} (\theta \cdot x^{o})$. However, by showing that for each $i=1, 2, 3, 4$, $$\sum_{k=1}^4  \a_{k,i}\sin \left(\frac{\a_k \cdot x \pi(l+1/2)}{\theta\cdot x^{o}}\right)=\sum_{k=1}^4 \a_{k,i}\cos \left(\frac{\a_k \cdot x \pi l}{\theta\cdot x^{o}}\right)=0, \quad l \in \mathbb{Z},$$ functions $H_l(x), \ K_l(x)$ are still differentiable, and hence we can conclude \eqref{3rd}. Since we need to show the equality for any $i=1,2,3,4$, we can simply assume $x=x^{o}$ without loss of generality. It can be easily checked that 
\begin{align*}
& \frac{\a_1 \cdot x^{o}}{\theta \cdot x^{o}}=\frac{\a_4 \cdot x^{o}}{\theta \cdot x^{o}}-4=\frac{3x^{(1)}-2x^{(2)}-x^{(4)}}{x^{(1)}-x^{(4)}}, \\
& \frac{\a_2 \cdot x^{o}}{\theta \cdot x^{o}}= \frac{\a_3 \cdot x^{o}}{\theta \cdot x^{o}}=\frac{2x^{(2)}-x^{(1)}-x^{(4)}}{x^{(1)}-x^{(4)}}, 
\end{align*}
and hence 
\begin{align*}
 \sin \left(\frac{\a_1 \cdot x \pi(l+1/2)}{\theta\cdot x^{o}}\right)&= \sin \left(\frac{\a_4 \cdot x \pi(l+1/2)}{\theta\cdot x^{o}}\right), \\ 
 \sin \left(\frac{\a_2 \cdot x \pi(l+1/2)}{\theta\cdot x^{o}}\right)&= \sin \left(\frac{\a_3 \cdot x \pi(l+1/2)}{\theta\cdot x^{o}}\right), \\
 \cos  \left(\frac{\a_1 \cdot x \pi l}{\theta\cdot x^{o}}\right)&= \cos  \left(\frac{\a_4 \cdot x \pi l}{\theta\cdot x^{o}}\right), \\
  \cos  \left(\frac{\a_2 \cdot x \pi l}{\theta\cdot x^{o}}\right)&= \cos  \left(\frac{\a_3 \cdot x \pi l}{\theta\cdot x^{o}}\right).
\end{align*}
We finish the proof of \eqref{3rd} by the following computation
\begin{align*}
\sum_{k=1}^4  \a_{k,i} \sin  \left(\frac{\a_k \cdot x \pi(l+1/2)}{\theta\cdot x^{o}}\right)=& \pm 2\left(\sin \left(\frac{\a_1 \cdot x \pi(l+1/2)}{\theta\cdot x^{o}}\right) -\sin \left(\frac{\a_2 \cdot x \pi(l+1/2)}{\theta\cdot x^{o}}\right)\right) \\
=& \pm4 \sin \left(\frac{2x^{(1)}-2x^{(2)}}{x^{(1)}-x^{(4)}}\pi(l+1/2)\right)\cos \left(\pi(l+1/2)\right) =0, \\
\sum_{k=1}^4 \a_{k,i}\cos  \left(\frac{\a_k \cdot x \pi l}{\theta\cdot x^{o}}\right) =& \pm 2 \left(\cos \left(\frac{\a_1 \cdot x \pi l}{\theta\cdot x^{o}}\right) -\cos \left(\frac{\a_2 \cdot x \pi l}{\theta\cdot x^{o}}\right)\right) \\
=& \pm4 \sin \left(\frac{2x^{(2)}-2x^{(1)}}{x^{(1)}-x^{(4)}}\pi l\right) \sin \left(\pi l \right)=0.
\end{align*}

{\it Step 5: Equation \eqref{2zero}.}
Finally, supposing $x=(x_1, x_1, x_1, x_1)$ and $x+\epsilon_j=(x_1+\epsilon)\mathbbm{1}_{\{k=j\}}+x_1\mathbbm{1}_{\{k \not =j\}} $, we calculate $\partial^2_{x_ix_j}u^T(t,x)$. According to \eqref{1st}, we have 
\begin{align*}
\pa_{x_i}u^T(t,x+\epsilon_j) = &\frac{1}{8\sqrt{2} }\int_{0}^{\infty} \frac{e^{-(T-t)r^2}}{r}\sum_{k=1}^4\a_{k,i}\left(\psi\left({ \frac{\epsilon r}{\sqrt{2}}+\frac{\pi}{2}}\right)\sin\left({r}{\a_k \cdot (x+\epsilon_j) }\right) \right. \notag \\
 &\left. + \psi\left( \frac{\epsilon r}{\sqrt{2}}\right)\cos\left({r}{\a_k \cdot (x+\epsilon_j)}\right)\right)dr+\frac{1}{4}. \notag 
\end{align*}
As a result of the equalities
\begin{align*}
\int_{\frac{1}{\sqrt{\epsilon}}}^{\infty} \frac{e^{-(T-t)r^2}}{r}dr=-\int_{\frac{1}{\sqrt{\epsilon}}}^{\infty} \frac{1}{2(T-t)r^2} d e^{-(T-t)r^2}=  \frac{\epsilon e^{- \frac{(T-t)}{ \epsilon}}}{2(T-t)}+\int_{\frac{1}{\sqrt{\epsilon}}}^{\infty} \frac{e^{-(T-t)r^2}}{4(T-t) r^3} dr, 
\end{align*}
we deduce that 
\begin{align}\label{2nd1}
0 =& \lim\limits_{\epsilon \to 0} \left( \frac{ e^{- \frac{(T-t)}{ \epsilon}}}{2(T-t)}+\int_{\frac{1}{\sqrt{\epsilon}}}^{\infty} \frac{e^{-(T-t)r^2}}{4(T-t) r}\right)6 \sqrt{2}  \geq \lim\limits_{\epsilon \to 0} \frac{6\sqrt{2}}{\epsilon} \left(  \frac{\epsilon e^{- \frac{(T-t)}{ \epsilon}}}{2(T-t)}+\int_{\frac{1}{\sqrt{\epsilon}}}^{\infty} \frac{e^{-(T-t)r^2}}{4(T-t) r^3}\right) \notag \\
\geq & \lim\limits_{\epsilon \to 0}  \frac{1}{\epsilon}  \int_{\frac{1}{\sqrt{\epsilon}}}^{\infty} \frac{e^{-(T-t)r^2}}{r}\left|\sum_{k=1}^4\a_{k,i}\left(\psi\left({ \frac{\epsilon r}{\sqrt{2}}+\frac{\pi}{2}}\right)\sin\left({r}{\a_k \cdot (x+\epsilon_j) }\right) \right. \right.\notag \\
 &\left. \left.+ \psi\left( \frac{\epsilon r}{\sqrt{2}}\right)\cos\left({r}{\a_k \cdot (x+\epsilon_j)}\right)\right)\right| dr.
\end{align}
From the equality 
\begin{align*}
\sum\limits_{k=1}^4 \a_{k,i}\cos(r \a_k \cdot (x+\epsilon_j))=&\mathbbm{1}_{\{i=j\}} \left(3\cos\left( \frac{3\epsilon r}{\sqrt{2}}\right)-3\cos\left( \frac{\epsilon r}{\sqrt{2}}\right) \right) \\
&+\mathbbm{1}_{\{i \not=j\}} \left(\cos\left( \frac{\epsilon r}{\sqrt{2}}\right) -\cos\left( \frac{3\epsilon r}{\sqrt{2}}\right)\right),
\end{align*}
it can be easily seen that 
\begin{align}\label{2nd2}
\lim\limits_{\epsilon \to 0} \frac{1}{r \epsilon} \sum\limits_{k=1}^4 \a_{k,i}\cos(r \a_k \cdot (x+\epsilon_j))=0.
\end{align}
Combining \eqref{2nd1} and \eqref{2nd2}, we conclude that 
\begin{align*}
\pa^2_{x_ix_j}&u^T(t,x)= \lim\limits_{\epsilon \to 0} \frac{1}{8\sqrt{2}\epsilon} \int_0^{\frac{1}{\sqrt{\epsilon}} }\frac{e^{-(T-t)r^2}}{r}\sum_{k=1}^4\a_{k,i}\left(\sin\left({r}{\a_k \cdot (x+\epsilon_j) }\right) +\cos\left({r}{\a_k \cdot (x+\epsilon_j)}\right)\right)dr \notag \\
=& \frac{1}{8\sqrt{2}} \lim\limits_{\epsilon \to 0}\int_0^{\frac{1}{\sqrt{\epsilon}} } e^{-(T-t)r^2}\left( \sum_{k=1}^4\a_{k,i}  \frac{\sin\left(r \a_k \cdot (x+\epsilon_j) \right) }{r \epsilon}+ \frac{1}{r \epsilon} \sum\limits_{k=1}^4 \a_{k,i}\cos(r \a_k \cdot (x+\epsilon_j)) \right) dr \\
=&\frac{1}{8\sqrt{2}}\int_0^{\infty} e^{-(T-t)r^2} \left( \sum\limits_{k=1}^4 \a_{k,i} \lim\limits_{\epsilon \to 0}   \frac{\sin\left(r \a_k \cdot (x+\epsilon_j) \right) }{r \epsilon} \right. \\
&+    \left.             \lim\limits_{\epsilon \to 0} \frac{1}{r \epsilon} \sum\limits_{k=1}^4 \a_{k,i}\cos(r \a_k \cdot (x+\epsilon_j)) \right) dr \\
=&\frac{1}{8\sqrt{2}} \int_0^{\infty} e^{-(T-t)r^2} \left( \sum\limits_{k=1}^4 \a_{k,i} \a_{k,j} \right) dr. 
\end{align*}
Since $\pa_{x_j}H_l(x), \pa_{x_j}K_l(x)$ defined in {\it Step 3} are continuous with respect to $x$,  and both series $\sum\limits_{l \in \mathbb{Z}} \pa_{x_j}H_l(x), \sum\limits_{l \in \mathbb{Z}} \pa_{x_j}K_l(x)$ converge uniformly in any bounded region of $x$, the second derivative $\pa^2_{x_ix_j}u^T(t,x)$ is also continuous, and hence we have proved that $u^T(t,x)$ is in $C^2([0,T)\times \mathbb{R}^4)$. 
\end{proof}

\subsubsection{Solution Property}\label{sss:property}
\begin{proof}
Supposing that  $\{i_1, i_2, i_3, i_4\}=\{1,2,3,4\}$ are subscripts such that $x_{i_1} \leq x_{i_2} \leq x_{i_3} \leq x_{i_4}$, we prove the equation $$\pa_t u^T(t,x) +\frac{1}{2} \left( \pa_{x_{i_2}}+\pa_{x_{i_4}} \right)^2 u^T(t,x)=0 .$$
Taking derivative with respect to $t$, we obtain that
\begin{align}\label{eq:pa_t}
\pa_t u^T(t,x)=&\frac{-1}{16\sqrt{2} }\int_{-\infty}^{\infty} e^{-(T-t)r^2}\left(\psi\left({{r}\theta\cdot x^{o}+\frac{\pi}{2}}\right)\sum_{k=1}^4 \cos\left({r}{\a_k \cdot x}\right)  \right. \notag \\
 &\left. - \psi\left({{r}\theta\cdot x^{o}}\right)\sum\limits_{k=1}^4\sin\left({r}{\a_k \cdot x}\right)\right)dr+\frac{1}{4\sqrt{2}}\int_{-\infty}^{\infty} e^{-(T-t)r^2} dr - \frac{1}{4} \sqrt{\frac{\pi}{2(T-t)}} \notag \\
 =&\frac{-1}{16\sqrt{2} }\int_{-\infty}^{\infty} e^{-(T-t)r^2}\left(\psi\left({{r}\theta\cdot x^{o}+\frac{\pi}{2}}\right)\sum_{k=1}^4 \cos\left({r}{\a_k \cdot x}\right)  \right. \notag \\
 &\left. - \psi\left({{r}\theta\cdot x^{o}}\right)\sum\limits_{k=1}^4\sin\left({r}{\a_k \cdot x}\right)\right)dr.
\end{align}
According to \eqref{2nd} and the equality $\pa_{x_{i_2}}\left( \theta \cdot x^{o} \right)+\pa_{x_{i_4}}\left( \theta \cdot x^{o} \right)=0$, the series part cancels out and we have
\begin{align*}
\left(\pa_{x_{i_2}}+\pa_{x_{i_4}}  \right)^2 u^T(t,x)=& \frac{1}{16\sqrt{2} }\int_{-\infty}^{\infty} {e^{-(T-t)r^2}}\sum_{k=1}^4(\a_{k,i_2}+\a_{k,i_4})^2 \left(\psi\left({{r}\theta\cdot x^{o}+\frac{\pi}{2}}\right)\cos\left({r}{\a_k \cdot x}\right) \right. \notag \\
& \left. - \psi\left({{r}\theta\cdot x^{o}}\right)\sin\left({r}{\a_k \cdot x}\right)\right)dr.\notag
\end{align*}
Since $(\a_{k,i_2}+\a_{k,i_4})^2=2$ for every $k=1, 2, 3, 4$, we conclude that 
\begin{align}
\frac{1}{2}\left(\pa_{x_{i_2}}+\pa_{x_{i_4}}  \right)^2 u^T(t,x)=& \frac{1}{16\sqrt{2} }\int_{-\infty}^{\infty} {e^{-tr^2}}\sum_{k=1}^4 \left(\psi\left({{r}\theta\cdot x^{o}+\frac{\pi}{2}}\right)\cos\left({r}{\a_k \cdot x}\right) \right. \notag \\
& \left. - \psi\left({{r}\theta\cdot x^{o}}\right)\sin\left({r}{\a_k \cdot x}\right)\right)dr=-\pa_t u^T(t,x). \label{eq:0101}
\end{align}
\end{proof}

\subsection{Proof of Theorem \ref{thm:optimality}}\label{ss:proof2}
\begin{proof} 
By using arguments similar to the proof of \eqref{eq:0101}, we have for $(j,k)=(1,3),(1,4),(2,3)$,
\begin{align*}
\pa_t u^T(t,x) +\frac{1}{2} \left( \pa_{x_{i_j}}+\pa_{x_{i_k}} \right)^2 u^T(t,x)=0.
\end{align*}
From \eqref{1st}, we obtain that $\sum\limits_{k=1}^4 \pa_{x_k} u^T(t,x)=1$, which implies $$\pa^2_{xx} u^T(t,x) \left(e_1+e_2+e_3+e_4 \right)=0.$$
Subsequently, for all $J \in P(4)$, we have that 
\begin{align*}
e_J^\top \pa^2_{xx} u^T(t,x) e_J - e_{J^c}^\top \pa^2_{xx} u^T(t,x) e_{J^c}=\left( e_J^\top -e_{J^c}^\top \right) \pa_{xx}^2 u^T(t,x) \left(e_J+e_{J^c} \right)=0.
\end{align*}
Therefore, it remains to show that the strategies $J \in \{\emptyset, \{i_1, i_2\}, \{i_1\},\{i_2\},\{i_3\},\{i_4\}\}$ are suboptimal, i.e., 
\begin{align*}
\pa_t u^T(t,x)+\frac{1}{2}\sup_{J \in P(N)}e_{J}^\top \pa^2_{xx} u^T(t,x)e_J \leq 0. 
\end{align*}
Since the second derivatives of $u^T(t,x)$ are continuous, we assume that $\theta \cdot x^o <0, x^{(2)} < x^{(3)}$ without loss of generality. First we introduce some notations, and simplify the expressions for $\pa_t u^T(t,x), \pa^2_{xx} u^T(t,x)$. Define 
\begin{align*}
S_k:=& \sqrt{T-t} \int_{-\infty}^{\infty} {e^{-(T-t)r^2}}\left(\psi\left({{r}\theta\cdot x^{o}+\frac{\pi}{2}}\right)\cos\left({r}{\a_k \cdot x^o}\right)  - \psi\left({{r}\theta\cdot x^{o}}\right)\sin\left({r}{\a_k \cdot x^o}\right)\right)dr, \\
L_k:=&\sqrt{T-t}\left( \sum_{l\in \mathbb{Z}}\frac{(-1)^le^{- \frac{(T-t)(\pi (l+1/2))^2}{(\theta\cdot x^{o})^2}}}{\theta\cdot x^{o}}\sin \left(\frac{\a_k \cdot x^{o} \pi(l+1/2)}{\theta\cdot x^{o}}\right)\right. \\
& \left. -\sum_{l\in \mathbb{Z}}\frac{(-1)^le^{- \frac{(T-t)(\pi l)^2}{(\theta\cdot x^{o})^2}}}{\theta \cdot x^{o}}\cos \left(\frac{\a_k \cdot x^{o} \pi l}{\theta\cdot x^{o}}\right) \right).  
\end{align*}
According to \eqref{eq:pa_t} and  \eqref{2nd}, it can be checked that 
\begin{align}
&\pa_t u^T(t,x)=\frac{-1}{16\sqrt{2(T-t)}}\sum\limits_{k=1}^4 S_k,  \label{eq:1expression}\\
& \pa^2_{x_{i_h}x_{i_j}} u^T(t,x)=\frac{1}{16\sqrt{2(T-t)}} \left(\sum\limits_{k=1}^4 \a_{k, h}\a_{k,j} S_k+2\sum\limits_{k=1}^4 \pa_{x_{i_j}}(\theta \cdot x^o) \a_{k,h}L_k \right),\label{eq:2expression}
\end{align}
where we use the fact that the $i_h$-th coordinate of $x$ is the $h$-th coordinate of $x^o$.

 Define $\tilde{T}:=-\frac{\sqrt{T-t}\pi}{2\theta \cdot x^o}$, $\beta_k:=\frac{\a_k \cdot x^{o}}{2\pi\sqrt{T-t}}, k=1,2,3,4$, and  
$$f(r)=e^{-r^2}, \quad F_k^1(r)=f(r)\cos(2\pi \beta_k r), \quad F_k^2(r)=f(r)\sin(2 \pi \beta_k r).$$ Their Fourier transforms are given respectively by 
\begin{align*}
\hat{f}(v)&:=\int_{-\infty}^{\infty} f(x) e^{ -2\pi i vx} dx= \sqrt{\pi} e^{-\pi v^2}, \\
\hat{F_k^1}(v)&:=\frac{\hat{f}\left(v-\beta_k \right)+\hat{f}\left(v+\beta_k \right) }{2}, \\
\hat{F_k^2}(v)&:=\frac{\hat{f}\left(v-\beta_k \right)-\hat{f}\left(v+\beta_k \right) }{2i}.
\end{align*}

By change of variables, we obtain
\begin{align*}
S_k=&  \int_{-\infty}^{\infty} {e^{-r^2}}\left(\psi\left({{r} \frac{\theta\cdot x^{o}}{\sqrt{T-t}}+\frac{\pi}{2}}\right)\cos\left({r} \frac{{\a_k \cdot x}}{\sqrt{T-t}}\right)  - \psi\left({{r}\frac{\theta\cdot x^{o}}{\sqrt{T-t}}}\right)\sin\left(r \frac{\a_k \cdot x}{\sqrt{T-t}}\right)\right)dr \\
=& \sum_{l \in \mathbb{Z}} (-1)^l \int_{(2l-1)\tilde{T}}^{(2l+1)\tilde{T}} F_k^1(r) dr  + \sum_{l \in \mathbb{Z}} (-1)^l \int_{2l \tilde{T}}^{(2l+2) \tilde{T}} F_k^2(r) dr. \\ 
\end{align*}
Since the functions $F_k^1$ are even and  $F_k^2$ are odd, we obtain that
\begin{align}\label{eq:defS}
 S_k= \int_{-\infty}^{\infty} F_k^1(r) dr - 2 \sum\limits_{l \in \mathbb{Z}} \int_{\tilde{T}}^{3\tilde{T}} F_k^1(4l\tilde{T}+r) dr + 2 \sum\limits_{l \in \mathbb{Z}} \int_{0}^{2\tilde{T}} F_k^2(4l\tilde{T}+r) dr . 
\end{align}
Also it can be seen that 
\begin{align}\label{eq:defL}
L_k=\frac{2\tilde{T}}{\pi}\left(\sum\limits_{l \in \mathbb{Z}} (-1)^l F_k^2((2l+1)\tilde{T})+\sum\limits_{l \in \mathbb{Z}} (-1)^l F_k^1(2l\tilde{T}) \right).
\end{align}

{\it Step 1: $J=\{i_1,i_2\}$.}
We prove the inequality
\begin{align}\label{1100}
\pa_t u^T(t,x) +\frac{1}{2} \left( \pa_{x_{i_1}}+\pa_{x_{i_2}} \right)^2 u(t,x) \leq 0 .
\end{align}
According to trigonometric formulas, we have the following equalities 
$$\sin \left(\frac{\a_1 \cdot x^o \pi(l+1/2)}{\theta\cdot x^{o}}\right)=\sin \left(\frac{\a_2 \cdot x^o \pi(l+1/2)}{\theta\cdot x^{o}}\right),$$ 
$$\sin \left(\frac{\a_3 \cdot x^o \pi(l+1/2)}{\theta\cdot x^{o}}\right)=\sin \left(\frac{\a_4 \cdot x^o \pi(l+1/2)}{\theta\cdot x^{o}}\right), $$
$$\cos \left(\frac{\a_1 \cdot x^o \pi l}{\theta\cdot x^{o}}\right)=\cos \left(\frac{\a_2 \cdot x^o \pi l}{\theta\cdot x^{o}}\right),$$
$$\cos \left(\frac{\a_3 \cdot x^o \pi l}{\theta\cdot x^{o}}\right)=\cos \left(\frac{\a_4 \cdot x^o \pi l}{\theta\cdot x^{o}}\right).$$
Therefore we have $L_1=L_2, L_3=L_4$.
Plugging in \eqref{eq:1expression}, \eqref{eq:2expression} and noting that $\pa_{x_{i_1}} (\theta \cdot x^{o})=\pa_{x_{i_2}} (\theta \cdot x^{o})=\frac{1}{\sqrt{2}}$, it can be checked that 
\begin{align*}
\pa_t u^T(t,x) + \frac{1}{2} \left( \pa_{x_{i_1}}+\pa_{x_{i_2}} \right)^2 u(t,x) =& \frac{1}{16\sqrt{2(T-t)}}\left(\sum\limits_{k=1}^4\left( \frac{1}{2}(\a_{k,1}+\a_{k,2})^2-1  \right)S_k \right) \\
& + \frac{1}{16\sqrt{2(T-t)}}(4L_1-4L_4)\\
=& \frac{1}{4\sqrt{2(T-t)}}(L_1-L_4).  
\end{align*}
Let us introduce $$\mu:= \frac{\pi \a_{1} \cdot x^o}{4 \theta \cdot x^{o}}+\frac{\pi}{4}, \ \ \ \nu:= \frac{\pi\a_{4} \cdot x^o}{4 \theta \cdot x^{o}}+\frac{\pi}{4}, \ \ \ \hat{\tau}:=\frac{(T-t)i\pi}{4 \left(\theta \cdot x^{o} \right)^2},$$
and the Jacobi-theta function $$\theta_3\left( z,\tau \right):=\sum\limits_{l=-\infty}^{\infty} \exp\left(\pi i l^2 \tau +2 i lz \right). $$
We rewrite sine and cosine terms as $$(-1)^l \sin \left(\frac{\a_{1} \cdot x^o \pi(l+1/2)}{\theta\cdot x^{o}}\right)=-Re \ e^{2 i (2l+1) \mu} ,$$
$$-(-1)^l \cos \left(\frac{\a_{1} \cdot x \pi l}{\theta\cdot x^{o}}\right)=- Re \ e^{4 i l \mu}. $$
Note that $\exp(\pi i l^2 \hat{\tau})=e^{-\frac{(T-t) (\pi l)^2 }{4(\theta \cdot x^o)^2}} $. Then according to the definition of $L_1$, we obtain that
\begin{align*}
\frac{L_1}{\sqrt{T-t}}=-\frac{Re \ \theta_3(\mu, \hat{\tau})}{\theta \cdot x^o}=-\frac{  \theta_3(\mu, \hat{\tau})}{\theta \cdot x^o},
\end{align*}
and 
\begin{align*}
\frac{L_4}{\sqrt{T-t}}=-\frac{  \theta_3(\nu, \hat{\tau})}{\theta \cdot x^o}.
\end{align*}
Therefore, we obtain 
\begin{align}\label{eq:Jacobi}
\pa_t u^T(t,x) + \frac{1}{2} \left( \pa_{x_{i_1}}+\pa_{x_{i_2}} \right)^2 u(t,x) =-\frac{1}{4\sqrt{2}}\frac{\theta_3(\mu, \hat{\tau})-\theta_3(\nu, \hat{\tau})}{\theta \cdot x^o},
\end{align}
and \eqref{1100} is equivalent to 
\begin{equation}\label{1100''}
\theta_3\left( \mu,\hat{\tau}\right)-\theta_3\left(\nu,\hat{\tau}\right) \leq 0.
\end{equation}
Taking $q=e^{i\pi \hat{\tau}} $, we have the infinite product representation for the Jacobi-theta function (see e.g. \cite{MR2723248})
\begin{align}
\theta_3\left(z,\hat{\tau}\right)=\prod\limits_{l=1}^{\infty} \left(1-q^{2l}\right)\left(1+2q^{2l-1} \cos(2  z) +q^{4l-2}\right).  \label{eq:inftyproduct}
\end{align}
By the definition of $\mu, \nu$, it can be easily checked that 
\begin{align*}
&\mu -\nu=\frac{\pi\left( x^{(1)}-x^{(4)}\right) }{x^{(1)}+x^{(2)}-x^{(3)}-x^{(4)} } \in (0, \pi), \\
&\mu + \nu =\frac{\pi \left( x^{(1)}-x^{(2)}-x^{(3)}+x^{(4)}\right) }{2\left(x^{(1)}+x^{(2)}-x^{(3)}-x^{(4)} \right)}+\frac{\pi}{2} \in (0,\pi), \\
\end{align*}
and hence $dist( \mu, \mathbb{Z}\pi)> dist(\nu, \mathbb{Z}\pi)$. Subsequently, we have $\cos( 2 \mu) \leq \cos(2  \nu)$, and therefore conclude \eqref{1100''} by \eqref{eq:inftyproduct}.

{\it Step 2: $J=\emptyset$.} According to the Poisson summation formula for Fourier transform (see e.g. \cite{MR1970295}), it can be seen that 
\begin{align*}
\sum\limits_{l \in \mathbb{Z}} F_k^1(4l\tilde{T}+r)  =\sum\limits_{l \in \mathbb{Z}} \frac{1}{4\tilde{T}} \hat{F_k^1}\left(\frac{l}{4\tilde{T}}\right)e^{ i2\pi \frac{l}{4\tilde{T}} r }, \\
\sum\limits_{l \in \mathbb{Z}} F_k^2(4l\tilde{T}+r)  =\sum\limits_{l \in \mathbb{Z}} \frac{1}{4\tilde{T}} \hat{F_k^2}\left(\frac{l}{4\tilde{T}}\right)e^{ i2\pi \frac{l}{4\tilde{T}} r } \ .
\end{align*}
Then according to \eqref{eq:defS},
\begin{align*}
S_k=& \hat{F_k^1}(0)-\frac{1}{2\tilde{T}} \sum\limits_{l \in \mathbb{Z}} \hat{F_k^1}\left(\frac{l}{4T} \right) \int_{\tilde{T}}^{3\tilde{T}} e^{i 2\pi \frac{l}{4\tilde{T}} r} dr +\frac{1}{2\tilde{T}} \sum\limits_{l \in \mathbb{Z}} \hat{F_k^2}\left(\frac{l}{4\tilde{T}} \right) \int_0^{2\tilde{T}} e^{i 2\pi \frac{l}{4\tilde{T}} r} dr \notag \\
=&\sum\limits_{l \in \mathbb{Z}} (-1)^l  \hat{F_k^1 }\left( \frac{2l+1}{4\tilde{T}}  \right)\frac{2}{(2l+1)\pi}-\sum\limits_{l \in \mathbb{Z}}   \hat{F_k^2 }\left( \frac{2l+1}{4\tilde{T}}  \right)\frac{2}{i(2l+1)\pi}  \notag \\
=&\sum\limits_{l \in \mathbb{Z}} (-1)^l\hat{f} \left(\frac{2l+1}{4\tilde{T}}+(-1)^{l+1} \beta_k \right) \frac{2}{(2l+1)\pi} \ .
\end{align*}
By direct computation, 
\[-\frac{1}{2\tilde{T}}\hat{F_k^1}\left(\frac{l}{4T} \right) \int_{\tilde{T}}^{3\tilde{T}} e^{i 2\pi \frac{l}{4\tilde{T}} r} dr =
\begin{cases}
\hat{F_k^1}(0), \quad &\text{if } l=0, \\
(-1)^{(l-1)/2} \hat{F_k^1 }\left( \frac{l}{4\tilde{T}}  \right)\frac{2}{l\pi}, \quad &\text{if } l \text{ is odd}, \\
0, \quad &\text{if } l \text{ is even},
\end{cases}
\]
\[\frac{1}{2\tilde{T}} \hat{F_k^2}\left(\frac{l}{4\tilde{T}} \right) \int_0^{2\tilde{T}} e^{i 2\pi \frac{l}{4\tilde{T}} r} dr =
\begin{cases}
-\hat{F_k^2 }\left( \frac{l}{4\tilde{T}}  \right)\frac{2}{il\pi}, \quad &\text{if } l \text{ is odd}, \\
0, \quad &\text{if } l \text{ is even}. 
\end{cases}
\]
Therefore, we obtain that 
\begin{align}\label{lastS}
S_k=&\sum\limits_{l \in \mathbb{Z}} (-1)^l  \hat{F_k^1 }\left( \frac{2l+1}{4\tilde{T}}  \right)\frac{2}{(2l+1)\pi}-\sum\limits_{l \in \mathbb{Z}}   \hat{F_k^2 }\left( \frac{2l+1}{4\tilde{T}}  \right)\frac{2}{i(2l+1)\pi}  \notag \\
=&\sum\limits_{l \in \mathbb{Z}} (-1)^l\hat{f} \left(\frac{2l+1+(-1)^{l+1} 4 \tilde{T}\beta_k}{4\tilde{T}} \right) \frac{2}{(2l+1)\pi} \notag \\
=&2\sum\limits_{l \geq 0} (-1)^l\hat{f} \left(\frac{2l+1+(-1)^{l+1} 4 \tilde{T}\beta_k}{4\tilde{T}} \right) \frac{2}{(2l+1)\pi},
\end{align}
where the last equation follows from the identity $$(-1)^l\hat{f} \left(\frac{2l+1+(-1)^{l+1} 4 \tilde{T}\beta_k}{4\tilde{T}} \right) \frac{2}{(2l+1)\pi}=(-1)^{-l-1}\hat{f} \left(\frac{-2l-1+(-1)^{-l} 4 \tilde{T}\beta_k}{4\tilde{T}} \right) \frac{2}{(-2l-1)\pi}.$$

Denoting 
\begin{align*}
&\eta_1:=4\tilde{T}\beta_1=-\frac{\a_1 \cdot x^o}{\theta \cdot x^o}=\frac{ 3x^{(1)}-x^{(2)}-x^{(3)}-x^{(4)} }{x^{(4)}+x^{(3)}-x^{(2)} -x^{(1)}} , \\
&\eta_4:=4\tilde{T}\beta_4=-\frac{\a_4 \cdot x^o}{ \theta \cdot x^o}=\frac{ -x^{(1)}-x^{(2)}-x^{(3)}+3x^{(4)} }{ x^{(4)}+x^{(3)}-x^{(2)} -x^{(1)}},
\end{align*}
it can be easily checked that they satisfy the constraints 
\begin{equation}\label{constraint}
\eta_1 \in [-3,-1], \ \ \ \eta_4 \in [1,3], \ \ \  \eta_4-\eta_1 \leq 4.
\end{equation}

Since $\pa_t u^T(t,x)=\frac{-1}{16\sqrt{2(T-t)}}\sum\limits_{k=1}^4 S_k$, the inequality 
\begin{align}
\pa_t u^T(t,x) \leq 0.
\end{align}
is equivalent to $\sum\limits_{k=1}^4 S_k \geq 0$. Due to definitions of $\tilde{T}$ and $\beta_k$,  we have that $4\tilde{T}\beta_1+4\tilde{T}\beta_2=-2, 4\tilde{T}\beta_3+4\tilde{T}\beta_4=2$. Therefore we obtain the equations 
\begin{align}\label{sum1}
S_1+S_2 =& 2 \sum\limits_{l \geq 0} (-1)^l\hat{f} \left(\frac{2l+1+(-1)^{l+1} \eta_1 }{4\tilde{T}}\right) \frac{2}{(2l+1)\pi}  \notag \\
& +2\sum\limits_{l \geq 0} (-1)^l \hat{f} \left(\frac{2l+1+(-1)^{l+1}(-2- \eta_1) }{4\tilde{T}}\right) \frac{2}{(2l+1)\pi} \notag \\
=&2\sum\limits_{l \geq 0} \left(\frac{2}{(4l+1)\pi}-\frac{2}{(4l+3)\pi} \right)\left( \hat{f}\left(\frac{4l+1-\eta_1}{4\tilde{T}} \right)+\hat{f}\left( \frac{4l+3+\eta_1}{4\tilde{T}} \right)  \right)  ,
\end{align}
\begin{align}\label{sum2}
 S_3+S_4 = & 2\sum\limits_{l \geq 0} (-1)^l\hat{f} \left(\frac{2l+1+(-1)^{l+1} \eta_4 }{4\tilde{T}}\right) \frac{2}{(2l+1)\pi}  \notag \\
& +2\sum\limits_{l \geq 0} (-1)^l \hat{f} \left(\frac{2l+1+(-1)^{l+1}(2- \eta_4) }{4\tilde{T}}\right) \frac{2}{(2l+1)\pi} \notag \\ 
=& -2\sum\limits_{l \geq 0}\left(\frac{2}{(4l-1)\pi}-\frac{2}{(4l+1)\pi} \right)\left( \hat{f}\left(\frac{4l-1+\eta_4}{4\tilde{T}} \right)+\hat{f}\left( \frac{4l+1-\eta_4}{4\tilde{T}} \right)  \right) .
\end{align}
It is obvious that $S_1+S_2 \geq 0$.  As a result of \eqref{constraint}, we obtain that $0 \leq -1+\eta_4 \leq 5-\eta_4 \leq 3+\eta_4 $, and hence the inequalities 
$$\hat{f} \left(-1+\eta_4 \right) \geq \hat{f}\left(5-\eta_4 \right)  \geq \hat{f}\left(3+\eta_4 \right).$$
Noting that $2\sum\limits_{l=1}^{\infty}\left(\frac{2}{(4l-1)\pi}-\frac{2}{(4l+1)\pi} \right)=\frac{4-\pi}{\pi}<\frac{4}{\pi}$, we get that
\begin{align*}
S_3+S_4=&  \frac{8}{\pi}\hat{f} \left(\frac{-1+\eta_4 }{4\tilde{T}}\right)  \\
&  -2\sum\limits_{l=1}^{\infty}\left(\frac{2}{(4l-1)\pi}-\frac{2}{(4l+1)\pi} \right)\left( \hat{f}\left(\frac{4l-1+\eta_4}{4\tilde{T}} \right)+\hat{f}\left( \frac{4l+1-\eta_4}{4\tilde{T}} \right)  \right) \\
\geq &  \frac{8}{\pi}\hat{f} \left(\frac{-1+\eta_4 }{4\tilde{T}}\right)-\frac{(4-\pi)}{\pi}\left(\hat{f} \left(\frac{3+\eta_4 }{4\tilde{T}}\right)+\hat{f} \left(\frac{5-\eta_4 }{4\tilde{T}}\right) \right) \geq 0,
\end{align*}
and hence $\sum\limits_{k=1}^4 S_k \geq S_3+S_4 \geq 0.$

{\it   Step 3: $J=\{i_h\}$.}
In the end, we prove the following inequality for each $h=1,2,3,4$, 
\begin{align}\label{1000}
-\pa_t u(t,x) +\frac{1}{2} \pa^2_{x_{i_h}x_{i_h}} u(t,x) \leq 0.
\end{align}
Recalling in \eqref{eq:defL}, we have 
$$ L_k=\frac{2\tilde{T}}{\pi}\left(\sum\limits_{l \in \mathbb{Z}} (-1)^l F_k^2((2l+1)\tilde{T})+\sum\limits_{l \in \mathbb{Z}} (-1)^l F_k^1(2l\tilde{T}) \right).$$
Applying Poisson summation formula, we obtain that 
\begin{align*}
\sum\limits_{l \in \mathbb{Z}} (-1)^l F_k^2((2l+1)\tilde{T})=&2\sum\limits_{l \in \mathbb{Z}} F_k^2((4l+1)\tilde{T})=\sum\limits_{l \in \mathbb{Z}} \frac{1}{2\tilde{T}} \hat{F_k^2}\left(\frac{l}{4\tilde{T}}\right)e^{ i\frac{\pi l}{2}}\\
=&\sum\limits_{l \in \mathbb{Z}} \frac{1}{2\tilde{T}} \hat{F_k^2}\left(\frac{2l+1}{4\tilde{T}}\right)e^{ i\frac{\pi (2l+1)}{2}}  \\
=& \frac{1}{4\tilde{T}}\sum\limits_{l \in \mathbb{Z}} (-1)^l  \left(\hat{f}\left(\frac{2l+1}{4\tilde{T}}-\beta_k \right)-\hat{f}\left(\frac{2l+1}{4\tilde{T}}+\beta_k \right) \right),
\end{align*}
\begin{align*}
\sum\limits_{l \in \mathbb{Z}} (-1)^l F_k^1(2l\tilde{T})&=-\sum\limits_{l \in \mathbb{Z}} F_k^1(2l\tilde{T})+2\sum\limits_{l \in \mathbb{Z}} F_k^1(4l\tilde{T})\\
&=-\sum\limits_{l \in \mathbb{Z}} \frac{1}{2\tilde{T}}\hat{F_k^1}\left(\frac{l}{2\tilde{T}}\right)+\sum\limits_{l \in \mathbb{Z}} \frac{1}{2\tilde{T}}\hat{F_k^1}\left(\frac{l}{4\tilde{T}}\right)\\
&=\frac{1}{4\tilde{T}}\sum\limits_{l \in \mathbb{Z}}  \left(\hat{f}\left(\frac{2l+1}{4\tilde{T}}-\beta_k \right)+\hat{f}\left(\frac{2l+1}{4\tilde{T}}+\beta_k \right) \right),
\end{align*}
and therefore 
\begin{align}\label{lastL}
L_k=\frac{1}{\pi}\sum\limits_{l \in \mathbb{Z}} \hat{f}\left(\frac{2l+1+(-1)^{l+1}4 \tilde{T}\beta_k}{4\tilde{T}} \right)=\frac{2}{\pi}\sum\limits_{l \geq 0} \hat{f}\left(\frac{2l+1+(-1)^{l+1}4\tilde{T} \beta_k }{4\tilde{T}}\right).
\end{align}

We first prove the following three inequalities by direct computation.
\begin{equation*}
S_1 \leq S_2, \ \ \  S_3 \leq  S_4, \ \ \ S_2 \leq S_4.
\end{equation*}
To prove the first inequality we write
\begin{align*}
S_2-S_1=& \sum\limits_{l \geq 0} (-1)^l \frac{4}{(2l+1)\pi} \left( \hat{f}\left(\frac{2l+1+(-1)^{l+1}(-2-\eta_1)}{4\tilde{T}} \right)- \hat{f}\left(\frac{2l+1+(-1)^{l+1}\eta_1}{4\tilde{T}} \right) \right) \\
=&\sum\limits_{l \geq 0} \left(\frac{4}{(4l+1)\pi}+\frac{4}{(4l+3)\pi} \right)\left( \hat{f}\left(\frac{4l+3+\eta_1}{4\tilde{T}} \right)- \hat{f}\left(\frac{4l+1-\eta_1}{4\tilde{T}} \right) \right).
\end{align*}
As a result of $0\leq 4l + 3+ \eta_1 \leq 4l+1-\eta_1$, we have for  every $l \geq 0$, $$ \hat{f}\left(\frac{4l+3+\eta_1}{4\tilde{T}} \right)- \hat{f}\left(\frac{4l+1-\eta_1}{4\tilde{T}} \right)  \geq 0,$$ and hence we conclude the first inequality. 

To show the second inequality we compute
\begin{align*}
S_4-S_3=& \sum\limits_{l \geq 0} (-1)^l \frac{4}{(2l+1)\pi} \left( \hat{f}\left(\frac{2l+1+(-1)^{l+1}\eta_4}{4\tilde{T}} \right)- \hat{f}\left(\frac{2l+1+(-1)^{l+1}(2-\eta_4)}{4\tilde{T}} \right) \right) \\
=&\sum\limits_{l \geq 1} \left(\frac{4}{(4l-1)\pi}+\frac{4}{(4l+1)\pi} \right)\left( \hat{f}\left(\frac{4l+1-\eta_4}{4\tilde{T}} \right)- \hat{f}\left(\frac{4l-1+\eta_4}{4\tilde{T}} \right) \right).
\end{align*}
It can be easily seen that $\hat{f}\left(\frac{4l+1-\eta_4}{4\tilde{T}} \right)- \hat{f}\left(\frac{4l-1+\eta_4}{4\tilde{T}} \right) \geq 0$ for any $l \geq 1$, and therefore we have proved the second inequality. 

Finally for the third inequality we have
\begin{align*}
S_4-S_2=& \sum\limits_{l \geq 0} (-1)^l \frac{4}{(2l+1)\pi} \left( \hat{f}\left(\frac{2l+1+(-1)^{l+1}\eta_4}{4\tilde{T}} \right)- \hat{f}\left(\frac{2l+1+(-1)^{l+1}(-2-\eta_1)}{4\tilde{T}} \right) \right). 
\end{align*}
For even $l \geq 0$, we have $|2l+1-\eta_4| \leq |2l+3+\eta_1| $, and hence $$\hat{f}\left(\frac{2l+1-\eta_4}{4\tilde{T}} \right)- \hat{f}\left(\frac{2l+3+\eta_1}{4\tilde{T}} \right)  \geq 0,$$ while for odd $l \geq 0$,  since $|2l+1+\eta_4| \geq |2l-1-\eta_1|$, we get that 
 $$\hat{f}\left(\frac{2l+1+\eta_4}{4\tilde{T}} \right)- \hat{f}\left(\frac{2l-1-\eta_1}{4\tilde{T}} \right)  \leq 0.$$
Subsequently we conclude the third inequality.

Now we prove \eqref{1000}. According to \eqref{eq:1expression}and \eqref{eq:2expression}, we have that 
\begin{align*}
\pa_t u^T(t,x) +\frac{1}{2} \pa_{x_{i_h}x_{i_h}} u^T(t,x)=\frac{1}{16\sqrt{2(T-t)}} \left( 2S_h-\frac{3}{4}(S_1+S_2+S_3+S_4) +L_1-L_4\right),
\end{align*}
and therefore the inequality is equivalent to 
\begin{align}\label{1000'}
L_1 -\frac{3}{4}(S_1+S_2) \leq L_4+\frac{3}{4}(S_3+S_4)-2S_h.
\end{align}
We have shown that $S_1 \leq S_2 \leq S_4, S_3 \leq S_4$. Subsequently, it is enough for us to prove the inequality for the case $h=4$. 
According to \eqref{lastS} and \eqref{lastL} can be checked that 
\begin{align*}
L_4+\frac{3}{4}S_3-\frac{5}{4} S_4=&\sum\limits_{l \geq 1} \left(\left(\frac{2}{\pi}+\frac{3}{(4l+1)\pi}+\frac{5}{(4l-1)\pi} \right)\hat{f}\left(\frac{4l-1+\eta_4}{4\tilde{T}} \right) \right. \\
& \left. + \left( \frac{2}{\pi}-\frac{3}{(4l-1)\pi}-\frac{5}{(4l+1)\pi}\right)\hat{f}\left(\frac{4l+1-\eta_4}{4\tilde{T}} \right) \right),
\end{align*}
\begin{align*}
L_1-\frac{3}{4}(S_1+S_2)=\sum\limits_{ l \geq 1} \left(\frac{2}{\pi}-\frac{3}{(4l+1)\pi}+\frac{3}{(4l+3)\pi} \right)\left(\hat{f}\left(\frac{4l+1-\eta_1}{4\tilde{T}} \right)+\hat{f}\left(\frac{4l+3+\eta_1}{4\tilde{T}} \right) \right).
\end{align*}
Note that $0 \leq 4l+1-\eta_4 \leq 4l-1+\eta_4 \leq 4l+3+\eta_1 \leq 4l+1-\eta_1$ for any $l \geq 1$. Subsequently we have that 
\begin{align*}
\hat{f}\left(\frac{4l+1-\eta_4}{4\tilde{T}} \right) \geq \hat{f}\left(\frac{4l-1+\eta_4}{4\tilde{T}} \right) \geq \hat{f}\left(\frac{4l+3+\eta_1}{4\tilde{T}} \right) \geq \hat{f}\left(\frac{4l+1-\eta_1}{4\tilde{T}} \right),
\end{align*}
and hence the inequalities
\begin{align*}
 \left(\frac{2}{\pi}-\frac{3}{(4l+1)\pi}+\frac{3}{(4l+3)\pi} \right) & \left(\hat{f}\left(\frac{4l+1-\eta_1}{4\tilde{T}} \right)+\hat{f}\left(\frac{4l+3+\eta_1}{4\tilde{T}} \right) \right) \\
 \leq &  \left(\frac{4}{\pi}-\frac{6}{(4l+1)\pi}+\frac{6}{(4l+3)\pi} \right) \hat{f}\left(\frac{4l+3+\eta_1}{4\tilde{T}} \right)  \\
 \leq & \frac{4}{\pi} \hat{f}\left(\frac{4l-1+\eta_4}{4\tilde{T}} \right)\\
 \leq & \left(\left(\frac{2}{\pi}+\frac{3}{(4l+1)\pi}+\frac{5}{(4l-1)\pi} \right)\hat{f}\left(\frac{4l-1+\eta_4}{4\tilde{T}} \right) \right. \\
& \left. + \left( \frac{2}{\pi}-\frac{3}{(4l-1)\pi}-\frac{5}{(4l+1)\pi}\right)\hat{f}\left(\frac{4l+1-\eta_4}{4\tilde{T}} \right) \right),
\end{align*} 
from which we conclude that $L_1-\frac{3}{4}(S_1+S_2) \leq L_4+\frac{3}{4}S_3-\frac{5}{4} S_4$ and also the inequality \eqref{1000'}.
\end{proof}
\subsection{Proof of Theorem \ref{thm:asyopt}}\label{ss:asymptotic}
\begin{proof}
The dynamics of state $X_m$ is given by
$$
X_{m}=X_{m-1}+e_{J_m}-\mathbbm{1}_{\{I_m \in J_m \}}\1.
$$
Take any sequence $m_M \in \mathbb{N}$ and $x_{m_M} \in \mathbb{R}^4$ such that $\frac{ m_MT}{M} \to t$, $\frac{x_{m_M} \sqrt{T}}{\sqrt{M}} \to x$ as $M \to \infty$. Denote $t_m=\frac{mT}{M}$, $\Delta X_{m}=e_{J_m}-\mathbbm{1}_{\{I_m \in J_m \}}\1$, and define the scaled state $$\tilde{x}_{m_M}=\frac{x_{m_M} \sqrt{T}}{\sqrt{M}}, \ \ \ \ \tilde{X}_m=\frac{X_m \sqrt{T}}{\sqrt{M}}, \ \ \ \ \Delta \tilde{X}_m=\frac{\Delta X_m \sqrt{T}}{\sqrt{M}}.$$ 

{\it Step 1: $\overline u^T(t,x) \leq u^T(t,x)$.}
To prove the inequality, we rewrite
\begin{align*}
\frac{\overline V^M(m_M,x_{m_M})\sqrt{T}}{\sqrt{M}}& - u^T\left(t,x \right) = \frac{  \sup_{\b\in \cV}\E^{\a^*(M),\b} \left[\Phi(X_M)|X_{m_M}=x_{m_M} \right]   \sqrt{T}  }{\sqrt{M}}-u^{T}(t,x)   \notag \\
= &  \sup_{\b\in \cV}\E^{\a^*(M),\b} \left[u^T(T, \tilde{X}_M)-u^T(t_{m_M}, \tilde{X}_{m_M}) |\tilde{X}_{m_M}=\tilde{x}_{m_M} \right]  \notag \\
&+u^T(t_{m_M},\tilde{x}_{m_M})-u^{T}(t,x)  \notag \\
=& \sum_{m=m_M+1}^{M} \sup_{\b \in \cV} \E^{\a^*(M),\b} \left[\left(u^T\left(t_m ,\tilde{X}_m\right)-u^T\left(t_{m-1},\tilde{X}_{m-1}\right) \right) |\tilde{X}_{m_M}=\tilde{x}_{m_M} \right] \notag \\
&+u^T(t_{m_M},\tilde{x}_{m_M})-u^{T}(t,x). \notag 
\end{align*}

Note that 
\begin{align}
&\E^{\a,\b}  \left[u^T\left(t_m,\tilde{X}_m\right)-u^T\left(t_{m-1},\tilde{X}_{m-1}\right)|\tilde{X}_{m-1}=\tilde{x}_{m-1} \right] \tag{$\star$}\label{eq:star} \\
&=\E^{\a,\b} \left[\pa_x u^T\left(t_{m-1},\tilde{x}_{m-1}\right)^\top\Delta \tilde{X}_{m} \right]  \label{eq:discrete0} \\
&\quad +2\E^{\a,\b} \left[\int_0^{\sqrt{\frac{T}{M}}} \left(\sqrt{\frac{T}{M}}-s\right)\left(\pa_t u^T+ \frac{1}{2}e_{J_{m}}^\top \pa_{xx}u^Te_{J_{m}}\right)(t_{m-1},\tilde{x}_{m-1}+s\Delta X_m) ds \right]  \label{eq:discrete1} \\
& \quad +2\E^{\a,\b}  \left[\int_0^{\sqrt{\frac{T}{M}}}  \left(\sqrt{\frac{T}{M}}-s\right)\left(\pa_t u^T(t_{m-1},\tilde{X}_m)-\pa_t u^T(t_{m-1},\tilde{x}_{m-1}+s\Delta X_m)  \right)  ds \right] \label{eq:discrete2} \\
&\quad +\E^{\a,\b}  \left[\int_0^{\frac{T}{{M}}} \left( \pa_t u^T(t_{m-1}+s, \tilde{X}_m)-\pa_t u^T(t_{m-1},\tilde{X}_m)\right)ds \right] \label{eq:discrete3}.
\end{align}

By the definition of $\a^*(M)$, the player chooses expert $i$ with probability $\pa_{x_i} u^T(t_{m-1},\tilde{x}_{m-1})$ at round $m$ for all $i=1,2,3,4$. Subsequently, we have 
\begin{align*}
\E^{\a^*(M),\b}\left[\pa_x u^T\left(t_{m-1},\tilde{x}_{m-1} \right)^\top\Delta \tilde{X}_m\right]&=\E^{\b}\left[\sum\limits_{i=1}^4\pa_{x_i} u^T\left( e_{J_m}^\top \pa_x u^T-\mathbbm{1}_{\{i \in J_m\}}  \1^\top \pa_x u^T \right) \right] \sqrt{\frac{T}{M} } \\
&=\E^{\b}\left[e_{J_m}^\top \pa_x u^T-\sum\limits_{i=1}^4\mathbbm{1}_{\{i \in J_m\}}   \pa_{x_i}u^T  \right]\sqrt{\frac{T}{M} }=0,
\end{align*}
where all the partial derivatives of $u^T$ are evaluated at $(t_{m-1}, \tilde{x}_{m-1})$. 

As a result of the solution property of $u$, the term \eqref{eq:discrete1} is non-positive. Also, it is easy to find the partial derivatives $\pa^2_{tt}u(t,x)$ and $\pa^2_{tx_i} u(t,x)$ 
\begin{align*}
\pa^2_{tt} u^T(t,x)=&\frac{-1}{16\sqrt{2} }\int_{-\infty}^{\infty} r^2e^{-(T-t)r^2}\left(\psi\left({{r}\theta\cdot x^{o}+\frac{\pi}{2}}\right)\sum_{k=1}^4 \cos\left({r}{\a_k \cdot x}\right)  \right. \notag \\
 &\left. - \psi\left({{r}\theta\cdot x^{o}}\right)\sum\limits_{k=1}^4\sin\left({r}{\a_k \cdot x}\right)\right)dr,
\end{align*}
\begin{align*}
\pa^2_{tx_i}u^T(t,x)=&\frac{1}{16\sqrt{2} }\int_{-\infty}^{\infty} r e^{-(T-t)r^2}\sum\limits_{k=1}^4 \a_{k,i} \left(\psi\left({{r}\theta\cdot x^{o}+\frac{\pi}{2}}\right) \sin \left({r}{\a_k \cdot x}\right)  \right. \notag \\
 &\left. + \psi\left({{r}\theta\cdot x^{o}}\right)\cos\left({r}{\a_k \cdot x}\right)\right)dr.
\end{align*}
According to the boundedness of $\psi, \sin, \cos$, we obtain that
\begin{align*}
\left|\pa^2_{tt} u^T(t,x) \right| \leq \frac{1}{16\sqrt{2}} \int_{-\infty}^{\infty}  8 r^2 e^{-(T-t)r^2} dr=\frac{1}{2\sqrt{2(T-t)^3}}  \int_{-\infty}^{\infty}  r^2 e^{-r^2} dr \leq \frac{C}{\sqrt{(T-t)^3}} ,
\end{align*}
\begin{align*}
\left|\pa^2_{tx_i} u^T(t,x) \right| \leq \frac{1}{16\sqrt{2}} \int_{-\infty}^{\infty} 6\sqrt{2} r e^{-(T-t)r^2} dr \leq \frac{6}{16(T- t)} \int_{-\infty}^{\infty} r e^{-r^2} dr  \leq \frac{C}{T-t},
\end{align*}
where $C$ is a positive constant independent of $(t,x)$ and is allowed to change from line to line.  

Noting that the above estimation is independent of $x$, we can therefore estimate the bound of \eqref{eq:discrete2} and \eqref{eq:discrete3}. 
\begin{align}
&\left|\E^{\a, \b} \left[ 2\int_0^{\sqrt{\frac{T}{M}}}  \left(\sqrt{\frac{T}{M}}-s\right)\left(\pa_t u^T(t_{m-1},\tilde{X}_m)-\pa_t u^T(t_{m-1},\tilde{x}_{n-1}+s\Delta X_m)  \right)  ds \right] \right|  \notag \\
& \quad \quad  \leq C \E^{\a, \b} \left[  \int_0^{\sqrt{\frac{T}{M}}}  \left(\sqrt{\frac{T}{M}}-s\right) ds \int_{s}^{\sqrt{\frac{T}{M} }} \left|  \pa^2_{tx}u^T\left(t_{m-1}, \tilde{x}_{m-1}+u\Delta X_m\right)  \right|  du \right] \notag \\
& \quad \quad \leq \frac{C}{T-t_{m-1}} \int_0^{\sqrt{\frac{T}{M}}}  \left(\sqrt{\frac{T}{M}}-s\right)^2 ds = \frac{C}{(T-t_{m-1})M^{\frac{3}{2}}} , \label{end1}
\end{align}
\begin{align}
& \left| \E^{\a, \b}\left[\int_0^{\frac{T}{{M}}} \left( \pa_t u^T(t_{m-1}+s, \tilde{X}_m)-\pa_t u^T(t_{m-1},\tilde{X}_m)\right)ds \right] \right| \notag \\
& \quad \quad \leq C\E^{\a, \b} \left[ \int_0^{\frac{T}{M}} ds \int_0^s  \left| \pa^2_{tt} u^T(t_{m-1}+u, \tilde{X}_m)  \right| du\right] \notag \\
& \quad \quad = C\E^{\a,\b} \left[\int_0^{\frac{T}{M}} \left(\frac{T}{M}-s \right)\left| \pa^2_{tt} u^T(t_{m-1}+s, \tilde{X}_m)  \right| ds \right] \notag \\
& \quad \quad \leq C \int_0^{\frac{T}{M}} \frac{\frac{T}{M}-s}{(T-t_{m-1}-s)^{\frac{3}{2}}} ds. \label{end2}
\end{align}

Therefore we obtain that 
\begin{align*}
& \sup_{\b \in \cV} \E^{\a^*(M),\b}  \left[u^T\left(t_m,\tilde{X}_m\right)-u^T\left(t_{m-1},\tilde{X}_{m-1}\right)|\tilde{X}_{m-1}=\tilde{x}_{m-1} \right] \\
& \quad \quad \leq C\left( \frac{1}{(T-t_{m-1})M^{\frac{3}{2}}}+ \int_0^{\frac{T}{M}} \frac{\frac{T}{M}-s}{(T-t_{m-1}-s)^{\frac{3}{2}}} ds \right).
\end{align*}
Let us estimate
\begin{align*}
\sum\limits_{m=1}^M \frac{1}{(T-t_{m-1})M^{\frac{3}{2}}}=\frac{1}{TM^{\frac{1}{2}} } \sum\limits_{k=1}^M \frac{1}{k} 
  \leq  \frac{1}{TM^{\frac{1}{2}} } \int_{1/2}^{M+1/2} \frac{1}{\lambda} d\lambda = \frac{\ln(M+1/2)-\ln(1/2) }{TM^{\frac{1}{2}}},
\end{align*}
and 
\begin{align*}
\sum\limits_{m=1}^M \int_0^{\frac{T}{M}} \frac{\frac{T}{M}-s}{(T-t_{m-1}-s)^{\frac{3}{2}}} ds=&\sum\limits_{m=1}^{M-1} \int_0^{\frac{T}{M}} \frac{\frac{T}{M}-s}{(T-t_{m-1}-s)^{\frac{3}{2}}} ds+  \int_0^{\frac{T}{M}} \frac{\frac{T}{M}-s}{(\frac{T}{M}-s)^{\frac{3}{2}}} ds \\
\leq & \sum\limits_{m=1}^{M-1} \frac{T}{M} \int_0^{\frac{T}{M}} \frac{1}{(T-t_{m-1}-s)^{\frac{3}{2}}} ds+\frac{1}{2}\sqrt{\frac{T}{M} } \\
= & \frac{T}{M}\int_{\frac{T}{M}}^T \frac{1}{s^{\frac{3}{2}} } ds  +\frac{1}{2}\sqrt{\frac{T}{M} }=\frac{T}{M} \left(\sqrt{\frac{M}{T} }-\frac{1}{\sqrt{T}} \right)+\frac{1}{2}\sqrt{\frac{T}{M} }.
\end{align*}
Thus we conclude that 
\begin{align}
& \lim\limits_{M \to \infty} \sum\limits_{m=1}^M \left( \frac{1}{(T-t_{m-1})M^{\frac{3}{2}}}+ \int_0^{\frac{T}{M}} \frac{\frac{T}{M}-s}{(T-t_{m-1}-s)^{\frac{3}{2}}} ds \right)  \notag \\
& \quad \leq \lim\limits_{M \to \infty} \left(\frac{T}{M} \left(\sqrt{\frac{M}{T} }-\frac{1}{\sqrt{T}} \right)+\frac{1}{2}\sqrt{\frac{T}{M} }+\frac{\ln(M+1/2)-\ln(1/2) }{TM^{\frac{1}{2}}} \right)=0, \label{eq:findalestimate}
\end{align}
and furthermore 
\begin{align*}
\overline u^T(t,x)- &u^T(t,x)  = \limsup\limits_{\left(M, \frac{m_MT}{M}, \frac{x_{m_M}\sqrt{T}}{\sqrt{M}} \right) \to (\infty, t,x) } \left( \frac{\overline V^M(m_M, x_{m_M})\sqrt{T} }{\sqrt{M}}-u^T(t,x)\right) \\
\leq &  \limsup\limits_{\left(M, \frac{m_MT}{M}, \frac{x_{m_M}\sqrt{T}}{\sqrt{M}} \right) \to (\infty, t,x) }  \sum_{m=m_M+1}^M C\left( \frac{1}{(T-t_{m-1})M^{\frac{3}{2}}}+ \int_0^{\frac{T}{M}} \frac{\frac{T}{M}-s}{(T-t_{m-1}-s)^{\frac{3}{2}}} ds \right)  \\
&+  \limsup\limits_{\left(M, \frac{m_MT}{M}, \frac{x_{m_M}\sqrt{T}}{\sqrt{M}} \right) \to (\infty, t,x) }  \left(u^T(t_{m_M},\tilde{x}_{m_M})-u^T(t,x) \right)=0.
\end{align*}

{\it Step 2: $\underline u^T(t,x) \geq u^T(t,x)$.} Similarly, we have 
\begin{align}
\frac{\underline V^M(m_M,x_{m_M})\sqrt{T}}{\sqrt{M}}& - u^T\left(t,x \right) \notag \\
=& \sum_{m=m_M+1}^{M} \inf_{\a \in \cU} \E^{\a,\cJ_{\cC}^b(M)} \left[\left(u^T\left(t_m ,\tilde{X}_m\right)-u^T\left(t_{m-1},\tilde{X}_{m-1}\right) \right) |\tilde{X}_{m_M}=\tilde{x}_{m_M} \right] \notag \\
&+u^T(t_{m_M},\tilde{x}_{m_M})-u^{T}(t,x), \notag 
\end{align}
and we need to estimate the conditional expectation \eqref{eq:star}.
At round $m$, the adversary chooses experts $\cJ_\cC(\tilde{x}_{m-1})$ with probability $\frac{1}{2}$, 
and $\cJ_\cC^c(\tilde{x}_{m-1})$ with probability $\frac{1}{2}$. Therefore we compute 
\begin{align}
\E^{\a,\cJ_\cC^b(M)}&\left[\pa_x u^T\left(t_{m-1},\tilde{x}_{m-1} \right)^\top\Delta \tilde{X}_m\right] \notag \\
&= \E^{\a}\left[\frac{1}{2}\left(e_{\cJ_\cC}^\top \pa_{x}u^T-\mathbbm{1}_{\{I_m \in \cJ_\cC \} }\1^\top  \pa_{x}u^T \right)+\frac{1}{2}\left(e_{\cJ^c_\cC}^\top \pa_{x}u^T-\mathbbm{1}_{\{I_m \in \cJ^c_\cC \} } \1^\top  \pa_{x}u^T \right)\right] \sqrt{\frac{T}{M} } \notag \\
&=\E^{\a}\left[\frac{1}{2}\left(\1^\top \pa_x u^T-\1^\top \pa_x u^T \right) \right]=0 .\label{end3}
\end{align}
Since the bounds of \eqref{eq:discrete2} and \eqref{eq:discrete3} are the same, it remains to find the lower bound of \eqref{eq:discrete1} when the adversary adopts the comb strategy. We show that if $J_m= \cJ_\cC(\tilde{x}_{m-1})$, then the following inequality holds
\begin{align}\label{eq:lowerbound}
\left(\pa_t u^T+ \frac{1}{2}e_{J_{m}}^\top \pa_{xx}u^Te_{J_{m}}\right)(t_{m-1},\tilde{x}_{m-1}+s\Delta X_m) \geq - \frac{Cs}{T-t_{m-1}},
\end{align}
where $C$ is a positive constant independent of $\tilde{x}_{m-1}$ and is allowed to change from line to line. The proof for the case $J_m=\cJ^c_\cC(\tilde{x}_{m-1})$ is the same. To simplify the notation, in the following argument, we denote $\cJ= \cJ_\cC(\tilde{x}_{m-1})$ and $\tilde{x}_s=\tilde{x}_{m-1}+s \Delta X_m$. 

Note that if $\tilde{x}_{m-1}^{(3)} \geq \tilde{x}_{m-1}^{(2)}+\sqrt{\frac{T}{M}}$, then according to Subsection \ref{sss:property}, we have that for any $s \in \left[0, \sqrt{\frac{T}{M} } \right]$
\begin{align}\label{eq:equal}
\left(\pa_t u^T+ \frac{1}{2}e_{\cJ}^\top \pa_{xx}u^Te_{\cJ}\right)(t_{m-1},\tilde{x}_s)=0,
\end{align}
which satisfies \eqref{eq:lowerbound}. Otherwise there exists a unique $s_0 \in \left[0,\sqrt{\frac{T}{M} } \right]$ such that $\tilde{x}_{s_0}^{(2)}=\tilde{x}_{s_0}^{(3)}$, i.e., $\tilde{x}_{m-1}^{(3)} = \tilde{x}_{m-1}^{(2)}+s_0$.
Then for $s \in [0,s_0]$, we still have \eqref{eq:equal}, but for $s \in \left[ s_0, \sqrt{ \frac{T}{M}}\right]$, according to the definition of $\cJ$, the adversary actually selects the first two leading experts. Recall \eqref{eq:Jacobi},  
\begin{align}\label{eq:auxiliary1}
\left(\pa_t u^T+ \frac{1}{2}e_{\cJ}^\top \pa_{xx}u^Te_{\cJ}\right)(t_{m-1},\tilde{x}_s)=\frac{\theta_3(\nu_s, \hat{\tau}_s )-\theta_3(\mu_s, \hat{\tau}_s) }{4\sqrt{2} \theta \cdot \tilde{x}^o_s}, 
\end{align}
where 
$$\mu_s:= \frac{\pi \a_{1} \cdot \tilde{x}^o_s}{4 \theta \cdot \tilde{x}_s^{o}}+\frac{\pi}{4}, \ \ \ \nu_s:= \frac{\pi\a_{4} \cdot \tilde{x}^o_s}{4 \theta \cdot \tilde{x}^o_s}+\frac{\pi}{4}, \ \ \ \hat{\tau}_s:=\frac{i\pi(T-t_{m-1})}{4 \left(\theta \cdot \tilde{x}^o_s\right)^2}.$$
Since $\tilde{x}^{(2)}_{s_0}=\tilde{x}^{(2)}_{s_0}$, it can be easily checked that $\mu_{s_0}-\nu_{s_0}=\pi$. According to the definition of Jacobi-theta function, $\theta_3(z+\pi,\tau)=\theta_3(z,\tau)$ and hence $\theta_3(\mu_{s_0}, \hat{\tau}_{s_0})=\theta_3(\nu_{s_0}, \hat{\tau}_{s_0})$. Let us calculate $\mu_s - \nu_s$ for $s \geq s_0$, 
\begin{align*}
\mu_s-\nu_s =\frac{\pi}{4} \left( \frac{ \a_{1} \cdot \tilde{x}^o_{s_0}-2(s-s_0)}{ \theta \cdot \tilde{x}_{s}^{o}}-\frac{\a_{4} \cdot \tilde{x}^o_{s_0}+2(s-s_0)}{ \theta \cdot \tilde{x}^o_{s}} \right)=\pi-  \frac{\pi(s-s_0)}{\theta \cdot \tilde{x}^o_s }.
\end{align*}
Then we have the estimation 
\begin{align}
\left| \theta_3(\nu_s, \hat{\tau}_s )-\theta_3(\mu_s, \hat{\tau}_s)  \right|  =& \left| \theta_3\left(\mu_s+\frac{\pi(s-s_0) }{\theta \cdot \tilde{x}^o_s}, \hat{\tau}_s \right)-\theta_3(\mu_s, \hat{\tau}_s) \right| \notag \\
= &\left|\sum\limits_{n=-\infty}^{\infty} \exp\left(i \pi \hat{\tau}_s n^2 \right) \left( \cos\left(2n\left( \mu_s+\frac{\pi(s-s_0) }{\theta \cdot \tilde{x}^o_s}\right) \right)- \cos(2n \mu_s)\right) \right| \notag \\
\leq & \left| \sum\limits_{n=-\infty}^{\infty} \exp\left(i \pi \hat{\tau}_s n^2 \right)  \frac{2n\pi(s-s_0)}{\theta \cdot \tilde{x}^o_s}\right|. \label{eq:auxiliary2}
\end{align}

To finish proofing \eqref{eq:lowerbound}, we need an auxiliary result 
\begin{align}\label{eq:auxiliary}
\sup\limits_{\lambda >0} \left( \sum\limits_{n=1}^{\infty} n\lambda e^{-n^2\lambda} \right) =C< +\infty. 
\end{align}
According to the inequality,
\begin{align*}
\sum\limits_{n=1}^{\infty} n\lambda e^{-n^2\lambda} \leq \sum\limits_{n=1}^{\infty} n\lambda e^{-n\lambda} = \frac{\lambda e^{\lambda}}{(e^{\lambda}-1 )^2},
\end{align*}
we conclude that $\lim\limits_{ \lambda \to \infty} \sum\limits_{n=1}^{\infty} n\lambda e^{-n^2\lambda}=0$ and $\lambda \mapsto \sum\limits_{n=1}^{\infty} n\lambda e^{-n^2\lambda}$ is continuous over  $\mathbb{R}_{>0}$. It remains to show that $\limsup\limits_{\lambda \to 0} \sum\limits_{n=1}^{\infty} n\lambda e^{-n^2\lambda} < \infty$. Fix $\lambda >0$, we can view $\sum\limits_{n=1}^{\infty} n\lambda e^{-n^2\lambda}$ as the Riemann sum of the integral $\int_1^{\infty} t\lambda e^{-t^2\lambda} dt$. It can be easily seen that $t \mapsto t\lambda e^{-t^2\lambda}$ is increasing over $\left[0, \frac{1}{\sqrt{2\lambda}}\right]$ and decreasing over $\left[\frac{1}{\sqrt{2\lambda}}, \infty \right]$. Take $I(\lambda)$ to be largest integer that is smaller than or equal to $\frac{ 1}{\sqrt{2\lambda}}$. 
Then we obtain that  
\begin{align*}
\sum\limits_{n=1}^{I(\lambda)-1} n\lambda e^{-n^2\lambda} & \leq \int_0^{I(\lambda)} t\lambda  e^{-t^2\lambda} dt ,\\
\sum\limits_{I(\lambda)+2}^{\infty } n\lambda e^{-n^2\lambda} & \leq \int_{I(\lambda)+1}^{\infty} t\lambda  e^{-t^2\lambda} dt,
\end{align*}
and therefore 
\begin{align*}
\sum\limits_{n=1}^{\infty} n\lambda e^{-n^2\lambda} & \leq I(\lambda)\lambda e^{-I(\lambda)^2\lambda}+ (I(\lambda)+1)\lambda e^{-(I(\lambda)+1)^2\lambda} + \int_1^{\infty} t\lambda e^{-t^2\lambda} dt \\
& \leq  I(\lambda)\lambda e^{-I(\lambda)^2\lambda}+ (I(\lambda)+1)\lambda e^{-(I(\lambda)+1)^2\lambda}+ \frac{1}{2}e^{-\lambda}.
\end{align*}
As a result of  $I(x)=\floor{\frac{1}{\sqrt{2\lambda}}}$, we conclude that 
\begin{align*}
\limsup\limits_{ \lambda \to 0} \sum\limits_{n=1}^{\infty} n\lambda e^{-n^2\lambda} \leq 2 \lim\limits_{\lambda \to 0} \frac{\lambda}{\sqrt{2\lambda}} e^{-\frac{1}{2}}+\lim\limits_{\lambda \to 0} \frac{1}{2}e^{-\lambda}=\frac{1}{2}.
\end{align*}

Taking $\lambda=\frac{ \pi^2(T-t_{m-1})}{4 (\theta \cdot \tilde{x}^o_s)^2}$ in \eqref{eq:auxiliary}, and combining \eqref{eq:auxiliary1},\eqref{eq:auxiliary2}, \eqref{eq:auxiliary}, we obtain that 
\begin{align*}
& \left(\pa_t u^T+ \frac{1}{2}e_{\cJ}^\top \pa_{xx}u^Te_{\cJ}\right) (t_{m-1},\tilde{x}_s) =\frac{\theta_3(\nu_s, \hat{\tau}_s )-\theta_3(\mu_s, \hat{\tau}_s) }{4\sqrt{2} \theta \cdot \tilde{x}^o_s} \\
& \quad \quad \geq -\sum\limits_{n=-\infty}^{\infty} \exp\left(i \pi \hat{\tau}_s n^2 \right)  \frac{2n\pi(s-s_0)}{(\theta \cdot \tilde{x}^o_s)^2} \geq -\frac{C(s-s_0)}{T-t_{m-1}} \sum\limits_{n=1}^{\infty} n \lambda e^{-n^2 \lambda} \geq -\frac{Cs}{T-t_{m-1}},
\end{align*}
 and hence 
 \begin{align*}
& \E^{\a,\cJ} \left[\int_0^{\sqrt{\frac{T}{M}}} \left(\sqrt{\frac{T}{M}}-s\right)\left(\pa_t u^T+ \frac{1}{2}e_{J_{m}}^\top \pa_{xx}u^Te_{J_{m}}\right)(t_{m-1},\tilde{x}_{m-1}+s\Delta X_m) ds \right]  \\
& \quad \quad \geq -\frac{C}{T-t_{m-1}} \int_0^{\sqrt{\frac{T}{M}}} \left(\sqrt{\frac{T}{M}}-s\right)s ds 
\geq \frac{-C}{(T-t_{m-1})M^{3/2} }.
 \end{align*}
 In conjunction with \eqref{end1}, \eqref{end2}, \eqref{eq:findalestimate} and \eqref{end3}, we obtain that 
\begin{align*}
& \inf\limits_{\a \in \cU}\E^{\a,\cJ_\cC^b(M)}  \left[u^T\left(t_m,\tilde{X}_m\right)-u^T\left(t_{m-1},\tilde{X}_{m-1}\right)|\tilde{X}_{m-1}=\tilde{x}_{m-1} \right] \\
& \quad \quad \geq - C\left( \frac{1}{(T-t_{m-1})M^{\frac{3}{2}}}+ \int_0^{\frac{T}{M}} \frac{\frac{T}{M}-s}{(T-t_{m-1}-s)^{\frac{3}{2}}} ds \right),
\end{align*}
and finally 
\begin{align*}
\underline u^T(t,x)- &u^T(t,x)  = \liminf\limits_{\left(M, \frac{m_MT}{M}, \frac{x_{m_M}\sqrt{T}}{\sqrt{M}} \right) \to (\infty, t,x) } \left( \frac{\underline V^M(m_M, x_{m_M})\sqrt{T} }{\sqrt{M}}-u^T(t,x)\right) \\
\geq &  \liminf\limits_{\left(M, \frac{m_MT}{M}, \frac{x_{m_M}\sqrt{T}}{\sqrt{M}} \right) \to (\infty, t,x) }  -\sum_{m=m_M+1}^M C\left( \frac{1}{(T-t_{m-1})M^{\frac{3}{2}}}+ \int_0^{\frac{T}{M}} \frac{\frac{T}{M}-s}{(T-t_{m-1}-s)^{\frac{3}{2}}} ds \right)  \\
&+  \liminf\limits_{\left(M, \frac{m_MT}{M}, \frac{x_{m_M}\sqrt{T}}{\sqrt{M}} \right) \to (\infty, t,x) }  \left(u^T(t_{m_M},\tilde{x}_{m_M})-u^T(t,x) \right)=0.
\end{align*}
\end{proof}

\subsection{Proof of Proposition \ref{conject}}\label{ss.conject}
\begin{proof}
The estimates in \cite[Theorem 4]{Drenska2019} allows us to claim that there exists a constant $C>0$ so that for all $T>0$ and $(t,x)\in [0,T]\times \R^N$, 
$$|u^T(t,x)| \leq C(T-t+1+|x|).$$
Thus, the function $v$ defined by the expression 
\eqref{eq:laplace}
\begin{align*}
v(x):=\E\left[\int_0^\infty e^{- T} \Phi(X^{0,x}_T)dT\right]=\int_0^\infty e^{- T}u^T(0,x)dT.
\end{align*}
has at most linear growth and due to \eqref{eq:checkconject} it is $C^2$. 
The optimality of comb strategies implies that for all $T>0$ and $(t,x)\in [0,T)\times \R^N$,
$$\frac{1}{2}\sup_{J \in P(N)}e_{J}^\top \pa^2_{xx} u^T(t,x)e_J=\frac{1}{2}e_{\cJ_\cC(x)}^\top \pa^2_{xx} u^T(t,x)e_{\cJ_\cC(x)}=-\pa_t u^T(t,x).$$
Equation \eqref{eq:checkconject} and the optimality of comb strategies imply that
\begin{align*}
&\frac{1}{2}\sup_{J \in P(N)}e_{J}^\top \pa^2_{xx} v(x)e_J\geq \frac{1}{2}e_{\cJ_\cC(x)}^\top \pa^2_{xx} ve_{\cJ_\cC(x)}\geq \int_0^\infty e^{- T}\frac{1}{2}e_{\cJ_\cC(x)}^\top \pa^2_{xx} u^T(0,x)e_{\cJ_\cC(x)}dT\\
&= \int_0^\infty e^{- T}\frac{1}{2}\sup_{J \in P(N)}e_{J}^\top \pa^2_{xx} u^T(0,x)e_JdT \geq \frac{1}{2}\sup_{J \in P(N)}e_{J}^\top \pa^2_{xx} v(x)e_J.
\end{align*}
Thus, using the fact that for some function $u^\sharp$, $u^T(t,x)=u^\sharp(T-t,x)$ for all $T>0$ and $(t,x)\in [0,T]\times\R^N$, all the inequalities above are equalities and 
$$\frac{1}{2}\sup_{J \in P(N)}e_{J}^\top \pa^2_{xx} v(x)e_J= -\int_0^\infty e^{- T}\pa_{t} u^T(0,x)dT=\int_0^\infty e^{- T}u^T(0,x)dT-u^0(0,x)=v(x)-\Phi(x).$$
Given the uniqueness of viscosity solution with linear growth for \eqref{eq:pdegeo2} proven in \cite[Theorem 5.1]{MR1118699} $v=u$ and comb strategies are indeed optimal for the problem \eqref{eq:pdegeo2}.
\end{proof}

\appendix

\section{Solutions from Inverse Laplace Transform}\label{appendix}
\subsection{A heuristic derivation for $N=4$} 
We derive the solution of \eqref{eq:pde} when there are $4$ experts. From \cite[Proposition 6.1]{2019arXiv190202368B}, the solution of the linear PDE 
\begin{align}\label{eq:pdegeo}
u(x)-\frac{1}{2}e_{\cJ_\cC(x)}^\top \pa^2 u(x)e_{\cJ_\cC(x)}= \Phi(x),
\end{align}
is given by 
\begin{align*} 
 u(x)=& x^{(4)}-\frac{\sqrt{2}}{4} \sinh(\sqrt{2}(x^{(4)}-x^{(3)}))+\frac{1}{4\sqrt{2}} \arctan\left(e^{\theta\cdot x^o}\right)\sum_{k=1}^4\cosh\left({\a_k \cdot x^{o}}\right)\\
\notag&+\frac{1}{4\sqrt{2}}\arctanh\left(e^{\theta\cdot x^o}\right)\sum_{k=1}^4\sinh\left({\a_k \cdot x^{o}}\right).
\end{align*}
It is well-known that an elliptic PDE can be solved by applying the Laplace transform to the corresponding parabolic one. Here, to obtain the solution to \eqref{eq:pdegeolinear}, we formally compute the inverse Laplace transform of \eqref{eq:pdegeo}. It can be easily checked that for $\lambda \in \mathbb{R}_{+}$
$$u^\lambda (x)=\lambda^{-3/2}u(\sqrt{\lambda }x)$$
solves the equation
$$\lambda u^\lambda(x)-\frac{1}{2} e_{\cJ_\cC(x)}^\top \pa_{xx} u^\lambda (x) e_{\cJ_\cC(x)}=\Phi(x).$$
We formally extend the function $\lambda \mapsto u^{\lambda}(x)$ to the complex plane with $\mathbb{R}_{-}$ as its branch cut. Applying the inverse Laplace transform for $t \in \mathbb{R}_{+}$, 
\begin{align}\label{eq:inverselaplace}
u^{\#}(t,x)=\frac{1}{2\pi i}\int_{x_0-i\infty}^{x_0+i\infty}e^{t\lambda}u^{\lambda}(x)d\lambda,
\end{align}
should solve the PDE, at least heuristically, 
\begin{align*}
\pa_t u^{\#}(t,x)-\frac{1}{2}e_{\cJ_\cC(x)}^\top \pa^2_{xx} u^{\#}(t,x)e_{\cJ_\cC(x)}=0, \notag\\
u^{\#}(0,x)=\Phi(x),
\end{align*}
 where $x_0$ is chosen so that the function to integrate is analytic on the line of integration. Now the solution of \eqref{eq:pdegeolinear} is given by $$u^T(t,x)=u^{\#}(T-t,x).$$

Let us compute \eqref{eq:inverselaplace}. Since the functions $\arctan, \arctanh$ can be extended to the complex plane via the formulas,
\begin{align*}
\arctan(z)=\frac{1}{2i} \log\left(\frac{i-z}{i+z} \right), \quad \arctanh(z)=\frac{1}{2} \log\left(\frac{1+z}{1-z} \right),
\end{align*}
we obtain that 
\begin{align}
\notag u(x)=& x^{(4)}-\frac{\sqrt{2}}{4} \sinh(\sqrt{2}(x^{(4)}-x^{(3)}))+\frac{1}{8i\sqrt{2}} \log\left(\frac{i-e^{\theta\cdot x^o}}{i+e^{\theta\cdot x^o}}\right)\sum_{k=1}^4\cosh\left({\a_k \cdot x^{o}}\right) \notag \\
&+\frac{1}{8\sqrt{2}}\log\left(\frac{1+e^{\theta\cdot x^o}}{1-e^{\theta\cdot x^o}}\right)\sum_{k=1}^4\sinh\left({\a_k \cdot x^{o}}\right).\label{eq:u} 
\end{align} To cancel the singularity at $\lambda =0$, we rewrite   
\begin{align}\label{eq:inverse}
u^{\#}(t,x)&=\frac{1}{2\pi i}\int_{x_0-i\infty}^{x_0+i\infty}e^{t\lambda}\frac{u(\sqrt{\lambda}x)-u(0)}{\lambda^{3/2}}d\lambda+\frac{u(0)}{2\pi i}\int_{x_0-i\infty}^{x_0+i\infty}e^{t\lambda}\frac{1}{\lambda^{3/2}}d\lambda \notag \\
&=\frac{1}{2\pi i}\int_{x_0-i\infty}^{x_0+i\infty}e^{t\lambda}\frac{u(\sqrt{\lambda}x)-u(0)}{\lambda^{3/2}}d\lambda+\frac{1}{2}\sqrt{\frac{t\pi}{2}},
\end{align}
where we use the facts that $u(0)=\frac{\pi}{4\sqrt{2}}$, and the inverse Laplace transform of $\frac{1}{ \lambda^{3/2}}$ is $\frac{2 \sqrt{t} }{\sqrt{\pi}}$. Take $x_0=1$, $\epsilon >0$, $R>0$, and the contour in Figure \ref{figure3}.
\begin{figure}
\centering
\begin{tikzpicture}
\def\gap{0.2}
\def\bigradius{3}
\def\littleradius{0.4}

\draw [help lines,->] (-1.2*\bigradius, 0) -- (1.2*\bigradius,0);
 \draw [help lines,->]   (0, -1.2*\bigradius) -- (0, 1.2*\bigradius);
\draw[black, thick,   decoration={ markings,
      mark=at position 0.17 with {\arrow{latex}}, 
      mark=at position 0.35 with {\arrow{latex}},
      mark=at position 0.53 with {\arrow{latex}},
      mark=at position 0.755 with {\arrow{latex}},  
      mark=at position 0.82 with {\arrow{latex}}, 
      mark=at position 0.955 with {\arrow{latex}}}, 
      postaction={decorate}]  
  let
     \n1 = {asin(\gap/2/\bigradius)},
     \n2 = {asin(\gap/2/\littleradius)},
     \n3={asin(0.5/\bigradius)}
  in (-180+\n1:\bigradius) arc (-180+\n1:-90+\n3:\bigradius)
  --(90-\n3:\bigradius) arc (90-\n3:180-\n1:\bigradius)
  -- (180-\n2:\littleradius) arc (180-\n2:-180+\n2:\littleradius)
  -- cycle;
  
\node at (-0.3,0.5) {$\gamma_{\epsilon}$};
\node at (-1.8,2.8) {$\gamma_1$};
\node at (-1.8,-2.8) {$\gamma_2$};
\node at (-1.2,0.29) {$l_1$};
\node at (-1.555,-0.32) {$l_2$};
\node at (5.3,2.5) {$\gamma_1:= \{Re^{i \theta}: \theta \in [\pi/2-\arcsin(1/R),  \pi-\arcsin({\epsilon}^2/R)]\} $}; 
\node at (5.4,1.5) {$\gamma_2:= \{Re^{i \theta}: \theta \in [\pi+\arcsin({\epsilon}^2/R), 3\pi/2+\arcsin(1/R)]\} $};
\node at (5.75,0.5) {$l_1:= \{(t , {\epsilon}^2): t \in [-R \cos\left(\arcsin({\epsilon}^2/R)\right), -  \epsilon \cos\left(\arcsin(\epsilon) \right) ]     \} $};
\node at (5.9,-0.5) {$l_2:= \{(t , -{\epsilon}^2): t \in[-R \cos\left(\arcsin({\epsilon}^2/R)\right), -  \epsilon \cos\left(\arcsin(\epsilon) \right)]      \} $};
\node at (4.6,-1.5) {$\gamma_{\epsilon}:= \{\epsilon e^{i \theta}: \theta \in[-\pi+\arcsin({\epsilon}), \pi-\arcsin({\epsilon})   ]  \} $};
\end{tikzpicture}
\caption{Contour of inverse Laplace transform}
\label{figure3}
\end{figure}
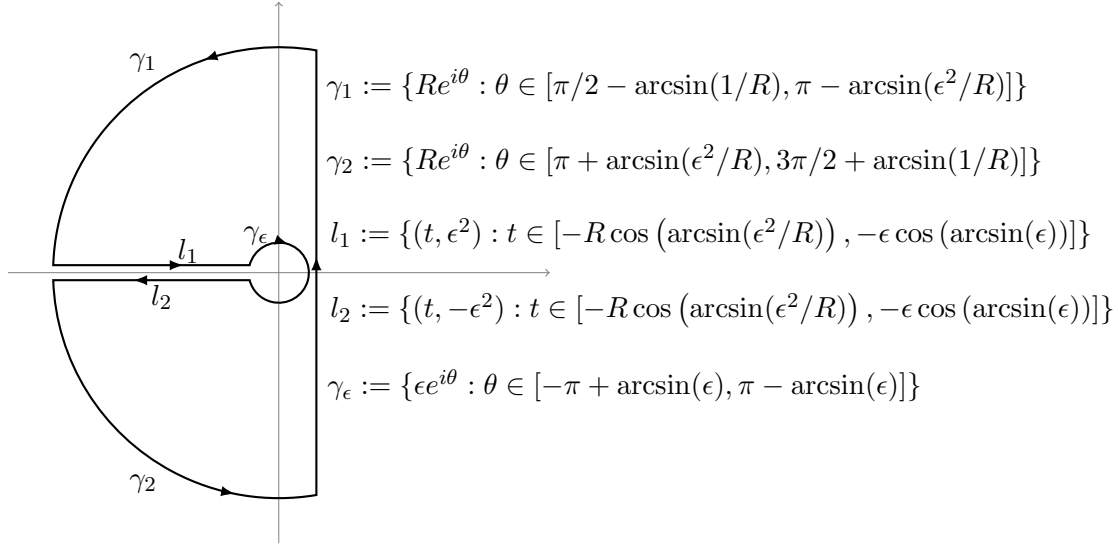
The integral of $e^{t \lambda} (u(\sqrt{ \lambda}x )-u(0)) /{\lambda}^{\frac{3}{2}} $ along the contour is zero. Letting $R \to \infty$, $\epsilon \to 0$, and assuming that the limit of the integral along $\gamma_1$, $\gamma_2$ vanish, we obtain that 
\begin{align*}
\frac{1}{2\pi i}\int_{x_0-i\infty}^{x_0+i\infty}e^{t\lambda}\frac{u(\sqrt{\lambda}x)-u(0)}{\lambda^{3/2}}d\lambda=-\lim\limits_{(R , \epsilon) \to (\infty, 0)} \frac{1}{2 \pi i } \int_{\gamma_{\epsilon}+l_1+l_2}e^{t\lambda}\frac{u(\sqrt{\lambda}x)-u(0)}{\lambda^{3/2}}d\lambda.
\end{align*}

It can be seen that 
\begin{align}\label{eq:branchpart}
&\lim\limits_{(R , \epsilon) \to (\infty, 0)} \frac{1}{2 \pi i } \int_{l_1+l_2}e^{t\lambda}\frac{u(\sqrt{\lambda}x)-u(0)}{\lambda^{3/2}}d\lambda  \notag \\
&\quad \quad  =\frac{1}{2\pi i}\int_{-\infty}^0 e^{t\lambda}\frac{u(\sqrt{\lambda}x)-u(0)}{\lambda^{\frac{3}{2}}} d\lambda+ \int_{0}^{-\infty} e^{t\lambda}\frac{u(\sqrt{\lambda}x)-u(0)}{\lambda^{\frac{3}{2}}} d\lambda, 
\end{align}
where the first integral is above the branch $\mathbb{R}_{-}$ and the second below. 
Thus the computation reduces to 
\begin{align*}
&\frac{1}{2\pi }\int_0^{\infty} e^{-tr}\frac{u(i\sqrt{r}x)+u(-i\sqrt{r}x)-2u(0)}{r^{\frac{3}{2}}} dr=\\
&\frac{1}{16\sqrt{2} }\int_0^{\infty} \frac{e^{-tr}}{r^{3/2}}\left(\frac{1}{i\pi} \left(\log\left(\frac{i-e^{i \sqrt{r}\theta\cdot x^o}}{i+e^{i \sqrt{r} \theta\cdot x^o}}\right)+\log\left(\frac{i-e^{-i \sqrt{r}\theta\cdot x^o}}{i+e^{-i \sqrt{r} \theta\cdot x^o}}\right)\right)\sum_{k=1}^4\cos\left(\sqrt{r}{\a_k \cdot x^{o}}\right)-4\right.\\
&\left.-\frac{1}{i\pi}\left(\log\left(\frac{1+e^{i\sqrt{r}\theta\cdot x^o}}{1-e^{i\sqrt{r}\theta\cdot x^o}}\right)-\log\left(\frac{1+e^{-i\sqrt{r}\theta\cdot x^o}}{1-e^{-i\sqrt{r}\theta\cdot x^o}}\right)\right)\sum_{k=1}^4\sin\left(\sqrt{r}{\a_k \cdot x^{o}}\right)\right)dr.
\end{align*}
For some values of $r$ depending on $x$, the first two $\log$ are respectively $\mp \infty$. But heuristically they cancel each other. Due to the factorizations 
\begin{align*}
\frac{i-e^{i \sqrt{r}\theta\cdot x^o}}{i+e^{i \sqrt{r} \theta\cdot x^o}}&=\frac{e^{i\left(\frac{\pi}{4}-\frac{\sqrt{r}\theta\cdot x^o}{2}\right)}-e^{i\left(-\frac{\pi}{4}+\frac{\sqrt{r}\theta\cdot x^o}{2}\right)}}{e^{i\left(\frac{\pi}{4}-\frac{\sqrt{r}\theta\cdot x^o}{2}\right)}+e^{i\left(-\frac{\pi}{4}+\frac{\sqrt{r}\theta\cdot x^o}{2}\right)}}=i\tan\left(\frac{\pi}{4}-\frac{\sqrt{r}\theta\cdot x^o}{2}\right),\\
\frac{1+e^{i\sqrt{r}\theta\cdot x^o}}{1-e^{i\sqrt{r}\theta\cdot x^o}}&=\frac{e^{i\left(-\frac{\sqrt{r}\theta\cdot x^o}{2}\right)}+e^{i\left(\frac{\sqrt{r}\theta\cdot x^o}{2}\right)}}{e^{i\left(-\frac{\sqrt{r}\theta\cdot x^o}{2}\right)}-e^{i\left(\frac{\sqrt{r}\theta\cdot x^o}{2}\right)}}=\frac{1}{-i \tan\left(\frac{\sqrt{r}\theta\cdot x^o}{2}\right)},
\end{align*}
and the identities 
\begin{align*}
\log\left(\frac{i}{ \tan\left(\frac{\sqrt{r}\theta\cdot x^o}{2}\right)}\right)-\log\left(\frac{-i}{ \tan\left(\frac{\sqrt{r}\theta\cdot x^o}{2}\right)}\right)&=i\pi sign\left(\tan\left(\frac{\sqrt{r}\theta\cdot x^o}{2}\right)\right),\\
\log\left(i\tan\left(\frac{\pi}{4}-\frac{\sqrt{r}\theta\cdot x^o}{2}\right)\right)+\log\left(i\tan\left(\frac{\pi}{4}+\frac{\sqrt{r}\theta\cdot x^o}{2}\right)\right)&=i\pi sign\left(\tan\left(\frac{\pi}{4}+\frac{\sqrt{r}\theta\cdot x^o}{2}\right)\right),
\end{align*}
it can be checked that the integral \eqref{eq:branchpart} becomes
\begin{align}
&\frac{1}{16\sqrt{2} }\int_0^{\infty} \frac{e^{-tr}}{r^{3/2}}\left(sign\left(\tan\left(\frac{\pi}{4}+\frac{\sqrt{r}\theta\cdot x^o}{2}\right)\right)\sum_{k=1}^4\cos\left(\sqrt{r}{\a_k \cdot x^{o}}\right)-4\right. \notag \\
&\left.- sign\left(\tan\left(\frac{\sqrt{r}\theta\cdot x^o}{2}\right)\right)\sum_{k=1}^4\sin\left(\sqrt{r}{\a_k \cdot x^{o}}\right)\right)dr. \label{eq:branchpart2}
\end{align}

According to \cite[Equation (3.4)]{2019arXiv190202368B}, we have $\lim\limits_{\lambda \to 0} \frac{u(\lambda x)-u(0) }{\lambda}=\frac{1}{4} \left(x_1+x_2+x_3+x_4 \right)$, and therefore 
\begin{align*}
\frac{1}{2 \pi i} \lim\limits_{\epsilon \to 0} \int_{\gamma_{\epsilon}} e^{t \lambda} \frac{u(\sqrt{\lambda}x )-u(0) }{{\lambda}^{\frac{3}{2}}} d\lambda=&\frac{-1}{2 \pi i} \lim\limits_{\epsilon \to 0} \int_0^{2 \pi} e^{t \epsilon e^{i\theta} }  \frac{u(\sqrt{\epsilon e^{i\theta}}x )-u(0) }{{\epsilon e^{i\theta}}^{\frac{1}{2}}}  i d \theta  \\
=& -\frac{1}{4}(x_1+x_2+x_3+x_4).
\end{align*}
In conjunction with \eqref{eq:inverse} and \eqref{eq:branchpart2}, we get that 
\begin{align*}
u^{\#}(t,x)=&\frac{-1}{16\sqrt{2} }\int_0^{\infty} \frac{e^{-tr}}{r^{3/2}}\left(sign\left(\tan\left(\frac{\pi}{4}+\frac{\sqrt{r}\theta\cdot x^o}{2}\right)\right)\sum_{k=1}^4\cos\left(\sqrt{r}{\a_k \cdot x^{o}}\right)-4\right. \notag \\
&\left.- sign\left(\tan\left(\frac{\sqrt{r}\theta\cdot x^o}{2}\right)\right)\sum_{k=1}^4\sin\left(\sqrt{r}{\a_k \cdot x^{o}}\right)\right)dr+ \frac{1}{4}\sum\limits_{i=1}^4 x_i + \frac{1}{2}\sqrt{\frac{ t \pi}{2} } \\
=& \frac{-1}{8\sqrt{2} }\int_0^{\infty} \frac{e^{-tr^2}}{r^2}\left(sign\left(\tan\left(\frac{\pi}{4}+\frac{r\theta\cdot x^o}{2}\right)\right)\sum_{k=1}^4\cos\left(r{\a_k \cdot x^{o}}\right)-4\right. \notag \\
&\left.- sign\left(\tan\left(\frac{r\theta\cdot x^o}{2}\right)\right)\sum_{k=1}^4\sin\left(r{\a_k \cdot x^{o}}\right)\right)dr+ \frac{1}{4}\sum\limits_{i=1}^4 x_i + \frac{1}{2}\sqrt{\frac{ t \pi}{2} },
\end{align*}
where the last equality follows from the change of variable. Since $u^T(t,x)=u^{\#}(T-t,x)$, we obtain \eqref{eq:solution}.

\subsection{Explicit expressions for $N=3$}\label{ss:derivation3}
According to \cite[Theorem 8]{Drenska2019}, the value function in the geometric stopping case is given by 
\begin{align}\label{eq:elliptic3}
u(x)=x^{(3)} + \frac{1}{2\sqrt{2}}e ^{\sqrt{2}(x^{(2)}-x^{(3)})} + \frac{1}{6\sqrt{2}}e^{\sqrt{2}(2x^{(1)}-x^{(2)}-x^{(3)})},
\end{align}
which solved \eqref{eq:pdegeo2} with $N=3$. 
We compute the inverse Laplace transform $$u^{\#}(t,x)=\frac{1}{2\pi i}\int_{x_0-i\infty}^{x_0+i\infty}e^{t\lambda}\lambda^{-3/2}u(\sqrt{\lambda}x)d\lambda,$$ 
where we extend the function $\lambda \mapsto u(\sqrt{ \lambda}x ) /{\lambda}^{\frac{3}{2}}$ naturally to $\mathbb{C} \setminus \mathbb{R}_{-}$.
The inverse Laplace transform of $\frac{1}{s}$, $\frac{1}{\sqrt{s}}$ and  $\frac{e^{-a\sqrt{s} } }{s}$ are $1$, $\frac{1}{\sqrt{\pi t}} $ and $erfc\left(\frac{a}{2\sqrt{t}}\right)$ respectively, where $erfc$ is the complementary error function (see e.g. \cite{MR1225604}). Subsequently according to the convolution theorem, it can be easily checked that
\begin{align*}
&u^{\#}(t,x)=x^{(2)}+\frac{1}{3}(2x^{(1)}-x^{(2)}-x^{(3)})+\frac{\sqrt{t}e^{\frac{-(2x^{(1)}-x^{(2)}-x^{(3)})^2}{2t}}}{3\sqrt{2\pi}}+ \frac{\sqrt{t}e^{\sqrt{t}\frac{-(x^{(2)}-x^{(3)})^2}{2t}}}{\sqrt{2\pi}}\\
&\quad \quad \quad \quad   -\frac{1}{3\sqrt{\pi}}(2x^{(1)}-x^{(2)}-x^{(3)})\int_{\frac{2x^{(1)}-x^{(2)}-x^{(3)}}{\sqrt{2t}}}^{\infty}e^{-y^2} dy-\frac{1}{\sqrt{\pi}}(x^{(2)}-x^{(3)})\int_{\frac{x^{(2)}-x^{(3)}}{\sqrt{2t}}}^{\infty} e^{-y^2} dy.
\end{align*}
Then $u^T(t,x):=u^{\#}(T-t,x)$ is our conjectured solution to \eqref{eq:pde} with $N=3$.

\begin{prop}\label{prop:solution}
The explicit solution to Equation \eqref{eq:pde} with $N=3$ is given by 
\begin{align}\label{eq:parobalic3}
&u^T(t,x):=x^{(2)}+\frac{1}{3}(2x^{(1)}-x^{(2)}-x^{(3)})+\frac{\sqrt{T-t}e^{\frac{-(2x^{(1)}-x^{(2)}-x^{(3)})^2}{2(T-t)}}}{3\sqrt{2\pi}} \notag \\
&\quad \quad \quad \quad+ \frac{\sqrt{T-t}e^{\sqrt{T-t}\frac{-(x^{(2)}-x^{(3)})^2}{2(T-t)}}}{\sqrt{2\pi}} -\frac{1}{3\sqrt{\pi}}(2x^{(1)}-x^{(2)}-x^{(3)})\int_{\frac{2x^{(1)}-x^{(2)}-x^{(3)}}{\sqrt{2(T-t)}}}^{\infty}e^{-y^2} dy \notag \\
&\quad \quad \quad \quad -\frac{1}{\sqrt{\pi}}(x^{(2)}-x^{(3)})\int_{\frac{x^{(2)}-x^{(3)}}{\sqrt{2(T-t)}}}^{\infty} e^{-y^2} dy.
\end{align}
\end{prop}
\begin{proof}
The proof follows from straightforward computations and is left to the reader.
\end{proof}

\bibliographystyle{siam}
\bibliography{ref}

\end{document}